\documentclass[11pt]{amsart}
\usepackage{amsfonts,amssymb,amsthm}
\usepackage{amsmath,amscd}
\usepackage{pstricks}
\usepackage{pstricks,pst-node}
\usepackage{mathrsfs}
\usepackage[all]{xy}

\DeclareMathAlphabet{\mathpzc}{OT1}{pzc}{m}{it}

\renewcommand{\subsection}[1]{\vspace{.18in}
\par\noindent\addtocounter{subsection}{1}
\setcounter{equation}{0}{\bf\thesubsection.\hspace{5pt}#1}}

\theoremstyle{definition}
\newtheorem{Rem}[subsection]{Remark}

\theoremstyle{plain}
\newtheorem{Prop}[subsection]{Proposition}
\newtheorem{Thm}[subsection]{Theorem}
\newtheorem{Not}[subsection]{Notation}
\newtheorem{Lem}[subsection]{Lemma}
\newtheorem{Coro}[subsection]{Corollary}
\numberwithin{equation}{subsection}

\newcommand{\ddHa}{\dot{{\mathfrak D}_\vtg}(n)_\sZ}
\newcommand{\ddbfHa}{\dot{\boldsymbol{\mathfrak D}_\vtg}(n)}

\newcommand{\afsfVn}{\mathsf V_\vtg(n)_\sZ}
\newcommand{\afbfsfVn}{\boldsymbol{\mathsf V}_\vtg(n)}

\newcommand{\Kbfj}{K^{\bfj}}
\newcommand{\afPin}{\Pi_\vtg(n)}
\newcommand{\tiafVnp}{{\mathcal V}_\vtg^+(n)}

\newcommand{\afVn}{\mathcal V_\vtg(n)_\sZ}
\newcommand{\afVnp}{\mathcal V_\vtg^+(n)_\sZ}
\newcommand{\afVnm}{\mathcal V_\vtg^-(n)_\sZ}
\newcommand{\afVnz}{\mathcal V_\vtg^0(n)_\sZ}
\newcommand{\afbfVn}{\boldsymbol{\mathcal V}_\vtg(n)}

\newcommand{\afTnrk}{\sT_\vtg(n,r)_{\mathpzc k}}

\newcommand{\afKn}{\mathsf K_\vtg(n)_\sZ}

\newcommand{\afbfKn}{\boldsymbol{\mathsf K}_\vtg(n)}
\newcommand{\afhbfKn}{{\hat{\boldsymbol{\mathsf K}}_\vtg(n)}}

\newcommand{\afbfSn}{{\boldsymbol{\mathcal S}}_\vtg(n)}
\newcommand{\afbfSno}{\boldsymbol{\mathcal S}_\vtg^\oplus(n)}

\newcommand{\tA}{{}^t\!A}

\newcommand{\han}{\subseteq}

\newcommand{\bsa}{{\boldsymbol{a}}}
\newcommand{\bsb}{\boldsymbol{b}}

\newcommand{\lan}{\langle}
\newcommand{\ran}{\rangle}
\newcommand{\dleb}{\left[\!\!\left[}
\newcommand{\leb}{\left[}
\newcommand{\drib}{\right]\!\!\right]}
\newcommand{\rib}{\right]}
\def\dblr#1{{[\![{#1}]\!]}}

\def\br#1{( #1 )}

\def\ddet#1{|\!| #1 |\!|}

\def\ggp#1#2{\left[\kern-3.2pt\left[{#1\atop #2}\right]\kern-3.2pt\right]}

\def\fS{{\frak S}}

\newcommand{\msP}{\mathscr P}

\newcommand{\afmsD}{{\mathscr D}^\vtg}

\newcommand{\affSr}{{\fS_{\vtg,r}}}

\newcommand{\afHr}{{\sH_\vtg(r)_\sZ}}
\newcommand{\afHrk}{{\sH_\vtg(r)_{\mpk}}}

\newcommand{\afbfHr}{{\boldsymbol{\mathcal H}_\vtg(r)}}

\def\sH{{\mathcal H}}

\def\aftisK{{\ti{\mathsf K}_\vtg(n)_{\sZ_1}}}

\def\sL{{\mathcal L}}
\def\sM{{\mathcal M}}

\def\sT{{\mathcal T}}
\def\sU{{\mathcal U}}

\def\sX{{\mathcal X}}

\def\sZ{{\mathcal Z}}

\newcommand{\vtg}{{\!\vartriangle\!}}
\newcommand{\Ha}{{{\mathfrak H}_\vtg(n)_{\sZ}}}

\newcommand{\bfHa}{{\boldsymbol{\mathfrak H}_\vtg(n)}}
\newcommand{\dbfHa}{{\boldsymbol{\mathfrak D}_\vtg}(n)}
\newcommand{\dbfHap}{{\boldsymbol{\mathfrak D}^+_\vtg}(n)}
\newcommand{\dbfHam}{{\boldsymbol{\mathfrak D}^-_\vtg}(n)}
\newcommand{\dbfHaz}{{\boldsymbol{\mathfrak D}^0_\vtg}(n)}
\newcommand{\dHa}{{{\mathfrak D}_\vtg}(n)_\sZ}

\newcommand{\dHak}{{{\mathfrak D}_\vtg}(n)_{\mpk}}
\newcommand{\dHap}{{{\mathfrak D}^+_\vtg}(n)_\sZ}

\newcommand{\dHam}{{{\mathfrak D}^-_\vtg}(n)_\sZ}
\newcommand{\dHaz}{{{\mathfrak D}^0_\vtg}(n)_\sZ}

\newcommand{\mpk}{\mathpzc k}

\def\field{{\mathbb F}}

\def\deg{{\rm deg}}

\newcommand{\Sal}{S_{\alpha}}
\newcommand{\mbzn}{\mathbb Z^{n}}

\newcommand{\mnmod}{\!\!\!\mod\!}

\newcommand{\tri}{\triangle(n)}

\newcommand{\afsl}{\widehat{\frak{sl}}_n}
\newcommand{\afgl}{\widehat{\frak{gl}}_n}

\newcommand{\afE}{E^\vartriangle}

\newcommand{\bbl}{\big[}
\newcommand{\bbr}{\big]}
\newcommand{\dbbl}{\big[\!\!\big[}
\newcommand{\dbbr}{\big]\!\!\big]}

\newcommand{\cycn}{{{n}}}
\newcommand{\afSr}{{\mathcal S}_{\vtg}(\cycn,r)_\sZ}
\newcommand{\afSpA}{{\mathcal S}_{\vtg}(n,pn+\sg(A))_\sZ}
\newcommand{\afSpB}{{\mathcal S}_{\vtg}(n,pn+\sg(B))_\sZ}
\newcommand{\afSpr}{{\mathcal S}_{\vtg}(n,pn+r)_\sZ}
\newcommand{\afSrk}{{\mathcal S}_{\vtg}(\cycn,r)_{\mathpzc k}}

\newcommand{\afsK}{{{\mathsf K}_\vtg(n)_\sZ}}
\newcommand{\afbfSr}{{\boldsymbol{\mathcal S}}_\vtg(\cycn,r)}

\newcommand{\afbse}{\boldsymbol e^\vartriangle}
\newcommand{\afmbnn}{\mathbb N_\vtg^{\cycn}}
\newcommand{\afmbzn}{\mathbb Z_\vtg^{\cycn}}

\newcommand{\afLa}{\Lambda_\vtg}
\newcommand{\afLanr}{\Lambda_\vtg(\cycn,r)}

\newcommand{\afThn}{\Theta_\vtg(\cycn)}
\newcommand{\aftiThn}{\widetilde\Theta_\vtg(\cycn)}

\newcommand{\afThnpm}{\Theta_\vtg^\pm(\cycn)}
\newcommand{\afThnp}{\Theta_\vtg^+(\cycn)}
\newcommand{\afThnm}{\Theta_\vtg^-(\cycn)}

\newcommand{\afThnr}{\Theta_\vtg(\cycn,r)}
\newcommand{\afTh}{\Theta_\vtg}

\newcommand{\afMnz}{M_{\vtg,\cycn}(\mathbb Z)}

\newcommand{\dzr}{\dot{\zeta}_r}
\newcommand{\hzr}{\h{\zeta}_r}
\newcommand{\hz}{\h\zeta}

\newcommand{\dxr}{\dot{\xi}_r}

\newcommand{\ttm}{\mathtt{m}}
\newcommand{\ttn}{\mathtt{m}}
\def\leq{\leqslant}\def\geq{\geqslant}
\def\ge{\geqslant}

\newcommand{\Th}{\Theta}
\newcommand{\dt}{\delta}

\newcommand{\vi}{\varphi}

\newcommand{\up}{v}

  \newcommand{\vep}{\varepsilon}
 \newcommand{\al}{\alpha}
 \newcommand{\bt}{\beta}
 \newcommand{\h}{\widehat}
 \newcommand{\ti}{\widetilde}

\newcommand{\zr}{\zeta_r}

\newcommand{\sg}{\sigma}
\newcommand{\Sg}{\Sigma}
\def\th{\theta}

\newcommand{\p}{\prec}
\newcommand{\pr}{\preccurlyeq}
\newcommand{\bop}{\bigoplus}
\newcommand{\op}{\oplus}
\newcommand{\ot}{\otimes}

\newcommand{\bfl}{\mathbf{0}}

\newcommand{\Ar}{{A,r}}
\newcommand{\mpm}{\mathpzc m}
\newcommand{\mpM}{\mathpzc M}

\newcommand{\ol}{\overline}
\newcommand{\lra}{\longrightarrow}
\newcommand{\ra}{\rightarrow}
 \newcommand{\la}{{\lambda}}

 \newcommand{\La}{\Lambda}

 \newcommand{\mbn}{\mathbb N}
 \newcommand{\mbq}{\mathbb Q}
 \newcommand{\mbc}{\mathbb C}
 \newcommand{\mbz}{\mathbb Z}

  \newcommand{\bfd}{{\mathbf{d}}}
 
 \newcommand{\bfj}{{\mathbf{j}}}

 \newcommand{\bfx}{{\mathbf{x}}}

\newcommand{\bfa}{{\boldsymbol{a}}}
\newcommand{\bfb}{{\boldsymbol{b}}}

\newcommand{\bfU}{{\mathbf{U}}}

\newcommand{\ga}{{\gamma}}

\newcommand{\End}{\operatorname{End}}
\newcommand{\Hom}{\operatorname{Hom}}

\newcommand{\res}{\operatorname{res}}
\newcommand{\spann}{\operatorname{span}}
\newcommand{\diag}{\operatorname{diag}}

\def\ro{\text{\rm ro}}
\def\co{\text{\rm co}}
\def\wh{\widehat}
\def\fkbfH{\boldsymbol{\mathfrak H}}

\newcommand{\bfsg}{{\boldsymbol\sigma}}

\def\afsygr{{\fS_{\vtg,r}}}

\def\dob{{\cdot\!\![}}
\def\dcb{{]\!\!\cdot}}
\def\dop{{\cdot\!\!(}}
\def\dcp{{)\!\!\cdot}}
\def\res{{\text{res}}}
\def\bfpU{{\text{\bf U}_\vtg(n)}}
\def\dbfpU{{\dot{\text{\bf U}}_\vtg(n)}}

\newcommand{\dbfUn}{\dot{{\bfU}}_\vtg(n)}
\newcommand{\dUn}{\dot{{U}_\vtg}(n)_\sZ}

\newcommand{\mpE}{\mathpzc E}

\newcommand{\aftiThnap}{\widetilde\Theta_\vtg^{\rm ap}(\cycn)}
\newcommand{\afThnrap}{\Theta_\vtg^{\rm ap}(\cycn,r)}

\begin{document}
\title{The Integral Quantum loop algebra of $\mathfrak{gl}_n$}

\author{Jie Du}
\address{School of Mathematics and Statistics, University of New South Wales,
Sydney 2052, Australia.} \email{j.du@unsw.edu.au}
\author{Qiang Fu$^\dagger$}
\address{Department of Mathematics, Tongji University, Shanghai, 200092, China.}
\email{q.fu@hotmail.com, q.fu@tongji.edu.cn}


\thanks{$^\dagger$Corresponding author.}
\thanks{Supported by the National Natural Science Foundation
of China, Fok Ying Tung Education Foundation and the Australian Research Council DP120101436.}

\begin{abstract}
We will construct the Lusztig form for the quantum loop algebra of $\mathfrak{gl}_n$ by proving the conjecture \cite[3.8.6]{DDF} and establish partially the Schur--Weyl duality at the integral level in this case. We will also investigate the integral form of the modified quantum affine $\mathfrak{gl}_n$ by introducing an affine stabilisation property and will lift the canonical bases from affine quantum Schur algebras to a canonical basis for this integral form. As an application of our theory, we will also discuss the integral form of the modified extended quantum affine $\mathfrak{sl}_n$ and construct its canonical basis to verify a conjecture of Lusztig in this case.
\end{abstract}
 \sloppy \maketitle
\section{Introduction}
Let $\sZ=\mbz[\up,\up^{-1}]$ be the integral Laurent polynomial ring.
It is well known that the Lusztig form $U_\sZ$ of a quantum enveloping $\mbq(v)$-algebra $\bfU$ associated with a Cartan matrix of finite or affine type is a $\sZ$-free subalgebra generated by divided powers of simple root vectors $E_{\al_i},F_{\al_i}$ together with group-like elements $K_{\al_i^\vee}^{\pm}$. In particular, there is a triangular decomposition
$U_\sZ=U^+_\sZ\cdot U^0_\sZ\cdot U^-_\sZ$ where, in the simply-laced case,
the 0-part $U_\sZ^0$ of this form is generated by $K_{\al_i^\vee}$ and $\bigl[{{K_{\al_i^\vee},0}\atop t}\bigr]$.

We now consider the quantum loop algebra $\bfU(\afgl)$. It contains a proper subalgebra
${}'\bfU=\bfU_\vtg(n)$ generated by $E_i=E_{\al_i},F_i=F_{\al_i}$ and $K_i^{\pm}$, $1\leq i\leq n$, where $
K_iK_{i+1}^{-1}=K_{\al_i^\vee}$ with $K_{n+1}=K_1$. This is called the ``extended'' quantum affine $\mathfrak{sl}_n$ in \cite{DDF} which is also investigated in \cite{Lu99} (cf. the definition in \cite[7.7]{Lu99}). Note that the subalgebra generated by $E_i,F_i$ and $K_{\al_i^\vee}$ is usually called the quantum enveloping algebra of affine
$\mathfrak{sl}_n$ type or the quantum loop algebra of $\mathfrak{sl}_n$ (see, e.g., \cite[9.3]{Lu99} or \cite[\S1.3]{DDF}).
If $'U_\sZ^+$ (resp., $'U_\sZ^-$) denotes the $\sZ$-subalgebra generated by divided powers $E_i^{(m)}$ (resp., $F_i^{(m)}$) and $U_\sZ^0$ denotes the $\sZ$-subalgebra generated by $K_i$ and $\bigl[{{K_i,0}\atop t}\bigr]$ ($t\in\mbn,1\leq i\leq n$), then the $\sZ$-submodule ${}'U_\sZ={}'U^+_\sZ\cdot U^0_\sZ\cdot{}'U^-_\sZ$ is a $\sZ$-free subalgebra of $'\bfU$ which is the Lusztig form of $'\bfU$ mentioned above. Now, naturally, one would ask what is a natural Lusztig form for $\bfU(\afgl)$?

By using Drinfeld's presentation for $\bfU(\afgl)$, a so-called restricted integral form $U^{\text{res}}_v(\afgl)$ was constructed over $\mbc[v,v^{-1}]$ by Frenkel--Mukhin in \cite[\S7.2]{FM}. However,  it is not clear from the construction whether $U^{\text{res}}_v(\afgl)$ is a Hopf algebra. Another integral form is constructed in \cite[2.4.4]{DDF} by using a double Ringel--Hall algebra presentation for $\bfU(\afgl)$. This integral form is the tensor product of the Lusztig form $'U_\sZ$ of $'\bfU$ with an integral central subalgebra. This is a Hopf subalgebra but not large enough to have integral affine quantum Schur algebras as its quotients; see example \cite[5.3.8]{DDF}.

However, there is a natural candidate constructed in \cite[\S3.8]{DDF}. By the double Ringel--Hall algebra presentation, we have a triangular decomposition: $\bfU(\afgl)\cong\dbfHa=\bfHa\cdot\bfU^0\cdot\bfHa^{\text{\rm op}}$,
where $\bfHa$ is a Ringel--Hall algebra over $\mbq(v)$ associated with a cyclic quiver and $\bfU^0=\mbq(v)[K_1^{\pm1},\ldots,K_n^{\pm1}]$ is the 0-part of $\bfU(\afgl)$. The candidate we proposed is to use the (integral) Ringel--Hall algebra $\Ha$ over $\sZ$ and the 0-part $U^0_\sZ$ defined above to form the $\sZ$-free submodule\footnote{It is denoted by $\ti{\mathfrak D}_\vtg(n)$ in \cite[(3.8.1.1)]{DDF}, while $\mathfrak D_\vtg(n)$ denote the tensor product of $'U_\sZ$ with the integral central subalgebra in \cite[2.4.4]{DDF}.} $\dHa:=\Ha\cdot U^0_\sZ\cdot\Ha^{\text{\rm op}}$.
We conjectured in \cite[3.8.6]{DDF} that $\dHa$ is a $\sZ$-subalgebra of $\dbfHa$. If the conjecture is true, then $\dHa$ is a Hopf subalgebra having integral affine quantum Schur algebras as its quotients.

In this paper, we will prove this conjecture. The proof is a beautiful application of a recent resolution of another conjecture, a realisation conjecture for quantum affine $\mathfrak{gl}_n$, by the authors \cite{DF13}, together with some successful attempts in the classical case \cite{Fu13, Fu} (see also \cite{Fu1}). The realisation conjecture is a natural affine generalisation of a new construction for quantum $\mathfrak{gl}_n$ via quantum Schur algebras by
A.A. Beilinson, G. Lusztig and R. MacPherson (BLM)  in \cite{BLM}.
This remarkable work has important applications to the investigation of integral quantum Schur--Weyl reciprocity
\cite{DPS}. This reciprocity at non-roots of unity was formulated in \cite{Jimbo} and its integral version was given in \cite{Du95,DPS}, built on the work \cite{BLM} and the Kazhdan--Lusztig cell theory.

Attempts to generalise the BLM work have been made by Ginzburg--Vasserot \cite{GV}, Lusztig \cite{Lu99}, etc. These constructions are geometric in nature, following BLM's geometric construction, but cannot resolve a realisation for the entire quantum affine $\mathfrak{gl}_n$. The main obstacle is that $\bfU(\afgl)$ cannot be generated by simple root vectors or simple generators. In \cite{DF13}, we discovered certain key multiplication formulas  by semisimple generators via the affine Hecke algebra and affine quantum Schur algebras. This allows, by modifying BLM's approach, to introduce a new algebra $\afbfVn$ by a basis together with explicit multiplication formuas of basis elements by semisimple generators. This algebra is isomorphic to $\dbfHa$ and hence to $\bfU(\afgl)$.

We now construct an integral $\sZ$-subalgebra $\afVn$ of $\afbfVn$ and then prove that the image of $\afVn$ in $\dbfHa$ coincides with $\dHa$. In this way we prove that $\dHa$ is a subalgebra. As an immediate application, the $\mbq(v)$-algebra epimorphism $\zeta_r$ given in \cite[Th.~3.8.1]{DDF}  restricts to a $\sZ$-algebra epimorphism $\zeta_{r}$ from $\dHa$ to the affine quantum Schur algebra $\afSr$. This establishes partially the Schur--Weyl duality at the integral level and, hence, at roots of unity.

There is another application of the key multiplication formulas mentioned above. In \cite{BLM}, the $\mbq(v)$-algebra $\boldsymbol{\sf K}(n)$ was constructed as a result of a stabilisation property. The algebra $\boldsymbol{\sf K}(n)$ is in fact isomorphic to the modified quantum group $\dot\bfU(\mathfrak{gl}_n)$. We will prove that
a stabilisation property continue to hold in the affine case. Thus, we may also introduce a new $\mbq(v)$-algebra $\afbfKn$, which is isomorphic to the modified quantum group $\dot\bfU(\afgl)$, and realise $\bfU(\afgl)$ as a subalgebra of the completion algebra $\afhbfKn$. In this way, we obtain an (unmodified!) affine generalisation of BLM's construction. We will further discuss the integral form $\afKn$ of $\afbfKn$ which is a realisation of $\ddHa$ (see Theorem \ref{realization of ddHa}) and construct its canonical basis as a lifting of the canonical bases for affine quantum Schur algebras. Applying our theory to the extended quantum affine $\mathfrak{sl}_n$, we will introduce the canonical basis for the modified quantum group $\dbfUn$ and verify in this case a conjecture of Lusztig \cite[9.3]{Lu99} which has been already proved in \cite{SV} (cf. \cite[7.9]{Mcgerty}).

The sections of the paper are organised as follows:
\begin{itemize}
\item[1.] Introduction
\item[2.] The double Ringel--Hall algebra presentation
\item[3.] A BLM type presentation
\item[4.] Some integral multiplication formulas
\item[5.] Lusztig form of $\bfU(\afgl)$  and integral affine quantum Schur--Weyl reciprocity
\item[6.] The affine BLM algebra $\afKn$
\item[7.] Canonical bases for the integral modified quantum affine $\frak{gl}_n$
\item[8.] Application to a conjecture of Lusztig.
\end{itemize}

\begin{Not}\label{aftiThn,afThn}\rm
For a positive integer $n$, let
$\afThn$  (resp., $\aftiThn$) be the set of all matrices
$A=(a_{i,j})_{i,j\in\mbz}$ with $a_{i,j}\in\mbn$ (resp. $a_{i,j}\in\mbz$, $a_{i,j}\geq0$ for all $i\neq j$)  such that
\begin{itemize}
\item[(a)]$a_{i,j}=a_{i+n,j+n}$ for $i,j\in\mbz$; \item[(b)] for
every $i\in\mbz$, both sets $\{j\in\mbz\mid a_{i,j}\not=0\}$ and
$\{j\in\mbz\mid a_{j,i}\not=0\}$ are finite.
\end{itemize}
Let $\afmbzn=\{(\la_i)_{i\in\mbz}\mid
\la_i\in\mbz,\,\la_i=\la_{i-n}\ \text{for}\ i\in\mbz\}$ and $\afmbnn=\{(\la_i)_{i\in\mbz}\in \afmbzn\mid \la_i\ge0\text{ for  }i\in\mbz\}.$ We will sometimes identify $\afmbzn$ with $\mbz^n$ via the  natural bijection $\flat:\afmbzn\lra\mbz^n$
defined by sending $\bfj$ to $\flat(\bfj)=(j_1,\cdots,j_n).$
Define an order relation $\leq$ and ``dot'' product on $\afmbzn$ by
\begin{equation}\label{order on afmbzn}
\la\leq\mu  \iff\la_i\leq \mu_i\,(1\leq i\leq n)\quad\text{ and }\quad \la\centerdot\mu=\la_1\mu_1+\cdot+\la_n\mu_n=
\flat(\la)\centerdot\flat(\mu).
\end{equation}
We say that $\la<\mu$ if $\la\leq\mu$ and $\la\not=\mu$.


Let $\mbq(\up)$ be the fraction field
of $\sZ=\mbz[\up,\up^{-1}]$. For integers $N,t$ with $t\geq 0$, define Gausian polynomials and their symmetric version in $\sZ$:
$
\dleb{N\atop t}\drib=\prod\limits_{1\leq
i\leq t}\frac{\up^{2(N-i+1)}-1}{\up^{2i}-1}
\,\,\text{ and }\,\,\leb{N\atop t}\rib=\up^{-t(N-t)}\dleb{N\atop t}\drib.
$
For $\mu\in\afmbzn$ and $\la\in\afmbnn$, let
$\dblr{\mu\atop\la} =\prod_{1\leq i\leq n}\dblr{\mu_i\atop\la_i}$ and let $[{\mu\atop\la}]=
\prod_{1\leq i\leq n}[{\mu_i\atop\la_i}]$. The following identity holds (see \cite[3.3]{Fu1}): for any $\la,\mu\in \afmbnn$ and $\al,\beta\in\afmbzn$,
\begin{equation}\label{SGPidentity}
\aligned
&\bigg[{\al+\beta\atop\la}\bigg]=\sum_{\mu\in\afmbnn,\mu\leq\la}v^{\al\centerdot(\la-\mu)-\mu\centerdot\beta}\bigg[{\al\atop\mu}\bigg]\bigg[{\beta\atop\la-\mu}\bigg];\\
&\bigg[{\al\atop\la}\bigg]\bigg[{\al\atop\mu}\bigg]=\sum_{{\gamma\in\afmbnn}\atop{\gamma\leq\la,\gamma\leq\mu}} v^{\la\centerdot\mu-\al\centerdot\gamma}\bigg[{\al\atop{\la+\mu-\gamma}}\bigg]\bigg[{{\la+\mu-\gamma}\atop{\gamma,\la-\gamma,\mu-\gamma}}\bigg];
\endaligned
\end{equation}
see \cite[Exercises 0.14 and 0.15]{DDPW} for a proof in the case of Gaussian polynomials.
\end{Not}

\section{The Double Ringel--Hall algebra presentation}

Let $\tri$ ($n\geq 2$) be
the cyclic quiver
with vertex set $I=\mbz/n\mbz=\{1,2,\ldots,n\}$ and arrow set
$\{i\to i+1\mid i\in I\}$. Note that we will regard $I$ as an abelian group as well as a subset of $\mbz$ depending on context.

Let $\field$ be a field. For $i\in I$ and $j\in\mbz$ with $i<j$, let $S_i$
denote the one-dimensional representation of $\tri$ with $(S_i)_i=\field$ and $(S_i)_k=0$ for $i\neq k$ and
$M^{i,j}$ the unique indecomposable nilpotent representation of dimension $j-i$ with top $S_i$.
Let
\begin{equation*}
\begin{split}
\afThnp&=\{A\in\afThn\mid a_{i,j}=0\text{ for }i\geq j\}.
\end{split}
\end{equation*}
\begin{Lem}
For any $A=(a_{i,j})\in\afThnp$, let
\begin{equation}\label{isoclass}
M(A)=M_\field(A)=\bop_{1\leq i\leq n,i<j}a_{i,j}M^{i,j}.
\end{equation}
Then $\sM=\{[M(A)]\}_{A\in\afThnp}$ forms a complete set of isomorphism classes of finite dimensional nilpotent representations of $\tri$.
\end{Lem}

Let $\bfd(A)\in\mbn I=\mbn^n$ be the dimension vector of $M(A)$.
For $\bfa=(a_i)\in\afmbzn$ and $\bfb=(b_i)\in\afmbzn$, the  Euler form  associated with the cyclic quiver $\tri$ is the bilinear form $\lan-,-\ran:\afmbzn\times\afmbzn\ra\mbz$ defined by
\begin{equation*}
\lan \bfa,\bfb\ran=\sum_{i\in I}a_ib_i-\sum_{i\in
I}a_ib_{i+1}.
\end{equation*}

By \cite{Ri93}, for $A,B,C\in\afThnp$,
the Hall polynomial $\vi^{C}_{A,B}\in\mbz[\up^2]$ is defined such
that, for any finite field $\field_q$,
$\vi^{C}_{A,B}|_{\up^2=q}$ is equal to the number of submodules $N$ of
$M_{\field_q}(C)$ satisfying $N\cong M_{\field_q}(B)$ and $M_{\field_q}(C)/N\cong M_{\field_q}(A)$.

The (generic) twisted {\it Ringel--Hall algebra} $\Ha$ of $\tri$ is, by definition, the $\sZ$-algebra spanned by basis
$\{u_A=u_{[M(A)]}\mid A\in\afThnp\}$ whose multiplication is defined by, for all $A,B\in \afThnp$,
$$u_{A}u_{B}=\up^{\lan \bfd (A),\bfd (B)\ran}\sum_{C\in\afThnp}\vi^{C}_{A,B}u_{C}.$$
Base change gives the $\mbq(v)$-algebra
$\bfHa=\Ha\otimes \mbq(\up)$.

We now describe the semisimple generators $u_\la=u_{[S_\la]}$ ($\la\in\afmbnn$) of $\Ha$, where $S_\la:=\op_{i=1}^n\la_iS_i$ is a semisimple representation of $\tri$.\footnote{For emphasising on semisimple generators,  we will use the same notation to denote the matrix in \eqref{isoclass} defining $S_\la$; see, e.g., Theorem
\ref{mul formulas in quantum affine gln}.}

On the set $\sM$ of isoclasses of finite dimensional nilpotent representations of
$\tri$, define a multiplication $*$ by $[M]*[N]=[M*N]$
for any $[M],[N]\in\sM$, where $M*N$ is the generic extension of $M$ by $N$.
By \cite{DD05,Re01} $\sM$ is a monoid
with identity $1=[0]$.

An element $\la$ in $\afmbnn$ is called {\it sincere} if $\la_i>0$ for all $i\in\mbz$. For $1\leq i\leq n$ let  $\afbse_i\in\afmbnn$ be the element
satisfying $(\afbse_i)_j=\dt_{\bar i,\bar j}$ for $j\in\mbz$. Here
$\bar i$ is the congruence class of $i$ modulo $n$.
Let
$$\widetilde I=\{\afbse_1,\afbse_2,\cdots,\afbse_n\}\cup\{\text{all sincere vectors in $\afmbnn$}\}.$$
Let $\ti\Sg$ be the set of words on the alphabet $\ti I$.

There is a natural surjective map $\wp^+:\ti\Sg\ra\afThnp$  (\cite[3.3]{DDX}) by taking $w=\bsa_1\bsa_2\cdots \bsa_m$ to $\wp^+(w)$, where $\wp^+(w)\in\afThnp$ is defined by
 $$[S_{\bsa_1}]*\cdots*[S_{\bsa_m}]=[M(\wp^+(w))].$$

For $A\in\afThnp$, let $$\ti u_A=\up^{\dim \End(M(A))-\dim M(A)}u_A.$$
For $\la\in\afmbnn$ let $\ti u_\la=\ti u_{[S_\la]}$. 
Any word $w=\bsa_1\bsa_2\cdots\bsa_m$ in $\ti\Sg$ can be uniquely expressed in the {\it tight form} $w=\bsb_1^{x_1}\bsb_2^{x_2}\cdots\bsb_t^{x_t}$ where $x_i=1$ if $\bsb_i$ is sincere, and $x_i$ is the number of consecutive occurrences of $\bsb_i$ if $\bsb_i\in\{\afbse_1,\afbse_2,\cdots,\afbse_n\}$.
For $w=\bsa_1\bsa_2\cdots\bsa_m\in\ti\Sg$ with the tight form
$\bsb_1^{x_1}\bsb_2^{x_2}\cdots\bsb_t^{x_t}$, define the associated monomials:
\begin{equation*}
\ti u_{(w)} =\ti u_{x_1\bsb_1}\ti u_{x_2\bsb_2}\cdots\ti u_{x_t\bsb_t}\in\Ha.
\end{equation*}

Following \cite[3.5]{BLM} and \cite{DF10} we may define the order relation $\pr$
on $\afMnz$ as follows.
For $A\in\afMnz$ and $i\not=j\in\mbz$, let\vspace{-2ex}
$$\sg_{i,j}(A)=\begin{cases}\sum\limits_{s\leq i,t\geq j}a_{s,t},\;&\text{ if $i<j$};\\
\sum\limits_{s\geq i,t\leq j}a_{s,t},\; & \text{ if
$i>j$}.\end{cases}\vspace{-2ex}$$
 For $A,B\in\afMnz$, define
\begin{equation}\label{order pref}
B\pr A \text{ if and only if } \sg_{i,j}(B)\leq\sg_{i,j}(A) \text{ for all } i\not=j.
\end{equation}
Put $B\p A$ if $B\pr A$ and, for some pair $(i,j)$ with $i\not=j$,
$\sg_{i,j}(B)<\sg_{i,j}(A)$.

Associated each $A\in\afThnp$ to a {\it distinguished word} $w_A$ (see \cite[(9.1)]{DDX}), the following {\it triangular relation} relative to $\pr$ between the monomial basis $\{\ti u_{(w_A)}\}_{A\in\afThnp}$ and the defining basis $\{\ti u_A\}_{A\in\afThnp}$ holds (see \cite[(9.2)]{DDX}, \cite[6.2]{DF10}):

\begin{Prop}\label{tri Hall}
For $A\in\afThnp$, there exist $w_{A}\in\ti\Sg$ such that $\wp^+(w_{A})=A$ and
\begin{equation}\label{eq tri Hall}
\ti u_{(w_{A})}=\ti u_A+\sum_{B\in\afThnp\atop B\prec A,\,\bfd(A)=\bfd(B)}f_{{B,A}}\ti u_B.\end{equation}
where $f_{{B,A^+}}\in\sZ$. In particular, $\Ha$ is generated by $u_\la$ for $\la\in\afmbnn$ with a monomial basis
$\{\ti u_{(w_{A})}\mid A\in\afThnp\}$.
\end{Prop}

The Hall algebra and its opposite algebra can be used to describe the $\pm$-part of quantum affine $\mathfrak{gl}_n$.
Let $\dbfHa$ be the (reduced) double Ringel--Hall algebra of the cyclic quiver  $\tri$ over $\mbq(\up)$ (cf. \cite{X97} and \cite[(2.1.3.2)]{DDF}). Then it has a triangular decomposition:
$$\dbfHa\cong\dbfHap\otimes\dbfHaz\otimes \dbfHam$$
with $\dbfHap= \fkbfH_\vtg(n)$, $\dbfHam= \fkbfH_\vtg(n)^{\text{op}}$, and $\dbfHaz=\mbq(\up)[K_1^{\pm1},\ldots,K_n^{\pm1}]$. We will add superscript $+$ or $-$ to $u_A$, $u_\la$, $u_{(w)}$, etc. for the corresponding
objects in $\dbfHap$ or $\dbfHam$.  Thus, $\dbfHa^\pm$ has basis $\{\ti u_A^\pm\}_{A\in\afThnp}$, generators $u_\la^\pm$, $\la\in\afmbnn$ and monomials $\ti u_{(w)}^\pm$.

Note that it is also natural to use the notation $\{\ti u_A:=\ti u^+_A\}_{A\in\afThnp}$ for a basis for $\dbfHap$ and the notation $\{\ti u_B:=\ti u^-_{^t\!B}\}_{B\in\afThnm}$ for the corresponding basis for $\dbfHam$, where
 $$\afThnm=\{A\in\afThn\mid a_{i,j}=0\text{ for }i\leq j\}.$$
With such notations, the matrix transpose induces the anti-isomorphism
\begin{equation}\label{tau}
 \tau:\dbfHap\lra\dbfHam, \quad\ti u_A\longmapsto \ti u_{\,{}^t\!\!A}.
 \end{equation}


For $A\in\aftiThn$, we write
\begin{equation}\label{A^+,A^-,A^0}
A=A^++A^0+A^-,\qquad A^\pm=A^++A^-
\end{equation}
 where $A^+\in\afThnp$,
$A^-\in\afThnm$ and $A^0$ is a diagonal matrix.

We have the following (not so elegant) presentation for quantum affine $\mathfrak{gl}_n$ via the double Ringel--Hall algebra; see
\cite[2.5.3, 2.6.1, 2.6.7 and 2.3.6(2)]{DDF}.
\begin{Thm} \label{presentation-dbfHa} {\rm(1)}
The (Hopf) algebra $\dbfHa$ is isomorphic to Drinfeld's algebra $\bfU(\afgl)$. It is the algebra over $\mbq(\up)$ which is spanned by basis
$$\{u_{A}^+K^{\bfj} u_{B}^-\mid A,B\in\afThnp,\bfj\in\afmbzn\}, \text{ where }
K^{\bfj}=K_1^{j_1}\cdots K_n^{j_n},$$ and is generated by
$u_\la^+$, $K_{i}^{\pm 1}$, $u_\mu^-$ $(\la,\mu\in\afmbnn,\,1\leq i\leq n)$, and whose multiplication is given by
the following relations:
\begin{itemize}
\item[(a)]
$K_iK_j=K_jK_i$, $K_iK_i^{-1}=1$;
\item[(b)]
$K^{\bfj} u_A^+=\up^{\langle{\bfd(A),\bfj}\rangle}u_A^+K^\bfj$,
$u_A^-K^\bfj=\up^{\langle{\bfd(A),\bfj}\rangle}K^\bfj u_A^-$;
\item[(c)]
$u_\la^+u_A^+=\sum_{C\in\afThnp}\up^{\lan \la,\bfd(A)\ran}\vi_{S_\la,A}^C u_C^+$;
\item[(d)]
$u_\mu^-u_A^-=\sum_{C\in\afThnp}\up^{\lan \bfd(A),\mu\ran}\vi_{A,S_\mu}^C u_C^-$;
\item[(e)] $u_\mu^-u_\la^+-u_\la^+ u_\mu^-=\displaystyle
\sum_{\al\not=0,\,\al\in\afmbnn\atop\al\leq\la,\,\al\leq\mu}\sum_{0\leq\ga\leq\al}
x_{\al,\ga}\ti K^{2\ga-\al} u_{\la-\al}^+u_{\mu-\al}^-,$ where the coefficients $x_{\al,\ga}\in\sZ$ are rather complicated as given in \cite[Cor.~2.6.7]{DDF}.
\end{itemize}

{\rm(2)} There exists a central subalgebra ${\mathbf Z}_\vtg(n)=\mbq(\up)[{\mathsf z}_m^+,{\mathbf z}_m^-]_{m\geq1}$ such that $\dbfHa\cong\bfU_\vtg(n)\otimes {\mathbf Z}_\vtg(n)$, where $\bfU_\vtg(n)$ is the subalgebra generated by $E_i=u^+_{\afbse_i}, F_i=u_{\afbse_i}^-, K_i$ for all $i\in I$.
\end{Thm}

We now define a candidate of the Lusztig form of $\dbfHa$.

Let $\dHap\cong\Ha$ (resp., $\dHam\cong\Ha^{\text{op}}$) be the $\sZ$-submodule of $\dbfHa$ spanned by the elements $u^+_A$ (resp., $u^-_A$) for $A\in\afThnp$, and let $\dHaz$ be the $\sZ$-subalgebra of $ \dbfHa$
generated by $K_i^{\pm1}$ and $\leb{K_i;0\atop t}\rib$, for $i\in I$
and $t\in\mbn$, where
$$\bigg[{K_i;0 \atop t}\bigg]=
\prod_{s=1}^t \frac{K_i\up^{-s+1}-K_i^{-1}\up^{s-1}}{\up^s-\up^{-s}}.$$
Let $\dHa=\dHap\dHaz\dHam$. We will prove in Theorem \ref{realization of dHa} that $\dHa$ is a $\sZ$-subalgebra of $\dbfHa$ and give a realisation for $\dHa$.

\section{A BLM type presentation}
We now describe a better presentation for $\dbfHa$.
Let $\afsygr$ be the group consisting of all permutations
$w:\mbz\ra\mbz$ such that $w(i+r)=w(i)+r$ for $i\in\mbz$.
The extended affine Hecke algebra $\afHr$ over $\sZ$ associated to
$\affSr$ is the (unital) $\sZ$-algebra with basis
$\{T_w\}_{w\in\affSr}$, and multiplication defined by
\begin{equation*}
\begin{cases} T_{s_i}^2=(\up^2-1)T_{s_i}+\up^2,\quad&\text{for }1\leq i\leq r\\
T_{w}T_{w'}=T_{ww'},\quad&\text{if}\
\ell(ww')=\ell(w)+\ell(w'),
\end{cases}
\end{equation*}
where $s_i\in\affSr$ is defined by
setting
$s_i(j)=j$ for $j\not\equiv i,i+1\mnmod r$, $s_i(j)=j-1$ for
$j\equiv i+1\mnmod r$ and $s_i(j)=j+1$ for $j\equiv i\mnmod r$. Let $\afbfHr=\afHr\ot_\sZ\mbq(\up)$.

For $\la=(\la_i)_{i\in\mbz}\in\afmbzn$
let
$
\sg(\la)=\sum_{1\leq i\leq n}\la_i$.
For $r\geq 0$  we set
\begin{equation*}
\afLanr=\{\la\in\afmbnn\mid\sg(\la)=r\}.
\end{equation*}
For $\la\in\afLanr$, let $\fS_\la:=\fS_{(\la_1,\ldots,\la_n)}$
be the corresponding standard Young subgroup of $\fS_r$.
For each $\la\in\afLanr$, let
$x_\la=\sum_{w\in\fS_\la}T_w\in\afHr$. The endomorphism algebras
$$\afSr:=\End_{\afHr}\biggl
(\bop_{\la\in\La_\vtg(n,r)}x_\la\afHr\biggr) {\text{ and }}\afbfSr:=\End_{\afbfHr}\biggl
(\bop_{\la\in\La_\vtg(n,r)}x_\la\afbfHr\biggr).$$
are called affine quantum Schur algebras (cf. \cite{GV,Gr99,Lu99}).

For $A\in\aftiThn$ and $r\geq 0$,
let
\begin{equation*}\sg(A)=\sum_{1\leq i\leq n,\,
j\in\mbz}a_{i,j}\quad\text{ and }\quad
\afThnr=\{A\in\afThn\mid\sg(A)=r\}.
\end{equation*}
For $\la\in\afLanr$, let
$\afmsD_\la=\{d\mid d\in\affSr,\ell(wd)=\ell(w)+\ell(d)\text{ for
$w\in\fS_\la$}\}$ and $\afmsD_{\la,\mu}=\afmsD_{\la}\cap{\afmsD_{\mu}}^{-1}$.
By \cite[7.4]{VV99} (see also \cite[9.2]{DF10}), there is
a bijective map
\begin{equation*}
{\jmath_\vtg}:\{(\la, d,\mu)\mid
d\in\afmsD_{\la,\mu},\la,\mu\in\afLanr\}\lra\afThnr
 \end{equation*}
sending $(\la, d,\mu)$ to the matrix $A=(|R_k^\la\cap dR_l^\mu|)_{k,l\in\mbz}$, where
\begin{equation*}
R_{i+kn}^{\nu}=\{\nu_{k,i-1}+1,\nu_{k,i-1}+2,\ldots,\nu_{k,i-1}+\nu_i=\nu_{k,i}\}\; \text{ with }\;\nu_{k,i-1}=kr+\sum_{1\leq t\leq i-1}\nu_t,
\end{equation*}
for all $1\leq i\leq n$, $k\in\mbz$ and $\nu\in\afLa(n,r)$.

For $\la,\mu\in\afLanr$ and $d\in\afmsD_{\la,\mu}$ satisfying $A=\jmath_\vtg(\la,d,\mu)\in\afThnr$, define
$e_A\in\afSr$ by
\begin{equation}\label{def of standard basis}
e_A(x_\nu h)=\dt_{\mu\nu}\sum_{w\in\fS_\la
d\fS_\mu}T_wh,
\end{equation}
where $\nu\in\afLanr$ and $h\in\afHr$, and let
\begin{equation}\label{dA}
[A]=\up^{-d_A}e_{A},\quad\text{ where } \quad
d_{A}=\sum_{1\leq i\leq n\atop i\geq k,j<l}a_{i,j}a_{k,l}.
\end{equation}
Note that the sets $\{e_A\mid A\in\afThnr\}$ and  $\{[A]\mid A\in\afThnr\}$ form $\sZ$-bases for $\afSr$.

Let
$$\afThnpm=\{A\in\afThn\mid a_{i,i}
=0\text{ for all $i$}\}.$$
For $A\in\afThnpm$, $\bfj\in\afmbzn$ and $\la\in\afmbnn$ let
\begin{equation*}
\begin{split}
A(\bfj,r)&=\sum_{\mu\in\afLa(n,r-\sg(A))}\up^{\mu\centerdot\bfj}
[A+\diag(\mu)]\in\afSr.\\
A(\bfj,\la,r)&=\sum_{\mu\in\afLa(n,r-\sg(A))}\up^{\mu\centerdot\bfj}
\leb{\mu\atop\la}\rib[A+\diag(\mu)]\in\afSr
\end{split}
\end{equation*}

The relationship between $\dbfHa$ and $\afbfSr$ can be seen from the following (cf. \cite{GV,Lu99} and \cite[Prop.~7.6]{VV99}).
\begin{Thm}[{\cite[3.6.3, 3.8.1]{DDF}}] \label{zr}
For $r\geq 0$, the map $\zr:\dbfHa\ra \afbfSr$ satisfying
$$\zr(K^\bfj)=0(\bfj,r),\;\zr(\ti u_A^+)=A(\bfl,r),\;\;\text{and}\;\;
\zr(\ti u_A^-)=(\tA)(\bfl,r),$$
for all $\bfj\in \afmbzn$, $A\in \afThnp$ and the transpose $\tA$ of $A$, is a surjective algebra homomorphism.
\end{Thm}

The map $\zr$ defined in Theorem \ref{zr} induce an algebra homomorphism
\begin{equation}\label{zeta}
\zeta=\prod_{r\geq 0}\zr:\dbfHa\ra\afbfSn.
\end{equation}

We now describe the image of $\zeta$.

Let $$\afbfSn=\prod_{r\geq 0}\afbfSr.$$
For $A\in\afThnpm$, $\bfj\in\afmbzn$ and $\la\in\afmbnn$, define elements in $\afbfSn$
\begin{equation*}
A(\bfj)=(A(\bfj,r))_{r\geq 0},\quad A(\bfj,\la)=(A(\bfj,\la,r))_{r\geq 0}.
\end{equation*}
We set, for $A\in\afMnz$ with $a_{i,j}<0$ for some $i\not=j$,
$A(\bfj,\la)=A(\bfj)=0$.

Let $\afbfVn$ be the $\mbq(\up)$-subspace of $\afbfSn$ spanned by
$A(\bfj,\la)$ for $A\in\afThnpm$, $\bfj\in\afmbzn$ and $\la\in\afmbnn.$  By \cite[Lem.~4.1]{DF13}, $\{A(\bfj)\mid A\in\afThnpm,\bfj\in\afmbnn\}$ forms a basis for $\afbfVn$.

\begin{Thm}[{\cite[4.4]{DF13}}]\label{realization}
The $\mbq(v)$-space $\afbfVn$ is a subalgebra of $\afbfSn$. Furthermore, the restriction of $\zeta$  to $\dbfHa$ induces a $\mbq(v)$-algebra isomorphism $\zeta:\dbfHa\ra\afbfVn$. In particular, we have
$$\zeta(K^\bfj)=0(\bfj),\;\zeta(\ti u_A^+)=A(\bfl),\;\;\text{and}\;\;
\zeta(\ti u_A^-)=(\tA)(\bfl),$$
for all $A\in\afThnpm$ and $\bfj\in\afmbzn$.
\end{Thm}

We shall identify $\dbfHa$ with $\afbfVn$ via the map $\zeta$ and identify $\dbfHa$ with $\bfU(\afgl)$ under the isomorphism given in Theorem~\ref{presentation-dbfHa}. The following better presentation for $\bfU(\afgl)$, called a {\it modified} BLM type realisation of quantum affine $\mathfrak{gl}_n$, is given in \cite[Th.~1.1]{DF13}.

For $T=(t_{i,j})\in\aftiThn$ let $\dt_T=(t_{i,i})_{i\in\mbz}\in\afmbzn,$
the ``diagonal'' of $T$ and let $\ti T=(\ti t_{i,j})$,
where $\ti t_{i,j}=t_{i-1,j}$ for all $i,j\in\mbzn$.

For $A\in\aftiThn$, let
$\ro(A)=\bigl(\sum_{j\in\mbz}a_{i,j}\bigr)_{i\in\mbz}$ and
$\co(A)=\bigl(\sum_{i\in\mbz}a_{i,j}\bigr)_{j\in\mbz}.$

\begin{Thm}\label{mul formulas in quantum affine gln}
The quantum loop algebra $\bfU(\afgl)$ is the $\mbq(\up)$-algebra which is { spanned by} the basis $\{A(\bfj)\mid A\in \Th^\pm_\vtg(n),\bfj\in\mbz^n_\vtg\}$ and { generated by} $0(\bfj)$,
 $\Sal(\bfl)$ and ${}^t\!\Sal(\bfl)$ for all $\bfj\in\afmbzn$ and $\al\in\afmbnn$, where $\Sal=\sum_{1\leq i\leq n}\al_i\afE_{i,i+1}$ and ${}^t\!\Sal$ is the transpose of $\Sal$,  and whose
{ multiplication rules} are given by:
\begin{itemize}
\item[(1)] $0(\bfj'){ A(\bfj)}=\up^{\bfj'\centerdot\ro(A)}{ A(\bfj'+\bfj)}$ and ${ A(\bfj)}0(\bfj')=\up^{\bfj'\centerdot\co(A)}{ A(\bfj'+\bfj)};$
\item[(2)] $\displaystyle \Sal(\bfl){ A(\bfj)}=
\sum_{T\in\afThn\atop\ro(T)=\al}\up^{f_{T}}\prod_{1\leq i\leq n\atop j\in\mbz,\,j\not=i}
\ol{\dleb{a_{i,j}+t_{i,j}-t_{i-1,j}\atop t_{i,j}}\drib}{ (A+T^\pm-\ti T^\pm)(\bfj_T,\delta_T)},$
\end{itemize}
where
$\bfj_T=\bfj+\sum_{1\leq i\leq n}(\sum_{j<i}(t_{i,j}-t_{i-1,j}))\afbse_i$ and
\begin{equation*}
\begin{split}
f_{T}&=\sum_{1\leq i\leq n\atop j\geq l,\,j\not=i}
a_{i,j}t_{i,l}-\sum_{1\leq i\leq n\atop j>l,\,j\not=i+1}a_{i+1,j}t_{i,l}
-\sum_{1\leq i\leq n\atop j\geq l,\,j\not=i}t_{i-1,j}t_{i,l}+\sum_{1\leq i\leq n\atop j>l,\,j\not=i,\,j\not=i+1}t_{i,j}t_{i,l}\\
&\qquad+\sum_{1\leq i\leq n\atop j<i+1}t_{i,j}t_{i+1,i+1}+\sum_{1\leq i\leq n}j_i(t_{i-1,i}-t_{i,i});
\end{split}
\end{equation*}
\begin{itemize}
\item[(3)] $\displaystyle {}^t\!\Sal(\bfl)A(\bfj)=
\sum_{T\in\afThn\atop\ro(T)=\al}v^{f_{T}'}\prod_{1\leq i\leq n\atop j\in\mbz,\,j\not=i}
\ol{\dleb{a_{i,j}-t_{i,j}+t_{i-1,j}\atop t_{i-1,j}}\drib}(A-T^\pm+\ti T^\pm)(\bfj'_T,\delta_{\ti T})$,
\end{itemize}
where $\bfj'_T=\bfj+\sum_{1\leq i\leq n}(\sum_{j>i}(t_{i-1,j}-t_{i,j}))\afbse_i$ and
\begin{equation*}
\begin{split}
f_{T}'&=\sum_{1\leq i\leq n\atop l\geq j,\,j\not=i}
a_{i,j}t_{i-1,l}-\sum_{1\leq i\leq n\atop l>j,\,j\not=i}a_{i,j}t_{i,l}
-\sum_{1\leq i\leq n\atop j\geq l,\,l\not=i}t_{i-1,j}t_{i,l}+\sum_{1\leq i\leq n\atop j>l,\,l\not=i,\,l\not=i+1}t_{i,j}t_{i,l}\\
&\qquad+\sum_{1\leq i\leq n\atop i<j}t_{i,j}t_{i-1,i}+\sum_{1\leq i\leq n}j_i(t_{i,i}-t_{i-1,i}).
\end{split}
\end{equation*}
\end{Thm}

\section{Some integral multiplication formulas}

 Let $\bar\ :\sZ\ra\sZ$ be the ring homomorphism defined by $\bar \up=\up^{-1}$.
The following result is proved in \cite[3.6]{DF13}.

\begin{Prop} \label{[B][A]}
Let  $A\in\afThnr$ and $\al,\ga\in\afmbnn$.

$(1)$ If $B\in\afThnr$ satisfies that $B-\sum\limits_{1\leq i\leq n}\al_i\afE_{i,i+1}$ is a diagonal matrix and $\co(B)=\ro(A)$, then in $\afSr$:
$$[B][A]=\sum_{T\in\afThn,\,\ro(T)=\al \atop
a_{i,j}+t_{i,j}-t_{i-1,j}\geq0,\,\forall i,j}v^{\bt(T,A)}\prod_{1\leq i\leq n\atop j\in\mbz}\ol{\dleb{a_{i,j}+t_{i,j}-t_{i-1,j}\atop t_{i,j}}\drib}[A+T-\ti T],$$
where $\bt(T,A)=\sum_{1\leq i\leq n,\,j\geq l}(a_{i,j}-t_{i-1,j})t_{i,l}-\sum_{1\leq i\leq n,\,j>l}(a_{i+1,j}-t_{i,j})t_{i,l}$.

$(2)$ If $C\in\afThnr$ satisfies that $C-\sum_{1\leq i\leq n}\ga_i\afE_{i+1,i}$ is a diagonal matrix and $\co(C)=\ro(A)$, then in $\afSr$:
$$[C][A]=\sum_{T\in\afThn,\,\ro(T)=\ga\atop
a_{i,j}-t_{i,j}+t_{i-1,j}\geq 0,\,\forall i,j}v^{\bt'(T,A)}\prod_{1\leq i\leq n\atop j\in\mbz}\ol{\dleb{a_{i,j}-t_{i,j}+t_{i-1,j}\atop t_{i-1,j}}\drib}[A-T+\ti T],$$
where $\bt'(T,A)=\sum_{1\leq i\leq n,\,l\geq j}(a_{i,j}-t_{i,j})t_{i-1,l}-\sum_{1\leq i\leq n,\,l>j}(a_{i,j}-t_{i,j})t_{i,l}$.
\end{Prop}

We now derive some integral version of the multiplication formulas.

\begin{Prop}\label{formula in Vz}
Let $A\in\afThnpm$, $\Sal=\sum_{1\leq i\leq n}\al_i\afE_{i,i+1}$ and
${}^t\!\Sal=\sum_{1\leq i\leq n}\al_i\afE_{i+1,i}$ with $\al\in\afmbnn$. Let $\la,\mu\in\afmbnn$, $\bfj,\bfj'\in\afmbzn$. The following identities holds in $\afbfSn$:
\begin{equation*}
\begin{split}
(1)\quad 0(\bfj',\mu)A(\bfj,\la)&=\sum_{\nu\in\afmbnn,\,\nu\leq\mu}a_\nu A(\bfj'+\bfj-\nu,\la+\mu-\nu);\qquad\qquad\qquad\qquad\qquad\qquad\qquad\quad\ \\ \end{split}
\end{equation*}
where
$$a_\nu=\sum_{\bfj''\in\afmbnn\atop\nu-\la\leq\bfj''\leq\nu}
v^{\ro(A)\centerdot(\bfj'+\mu-\bfj'')+\la\centerdot(\mu-\bfj'')}
\leb{\ro(A)\atop\bfj''}\rib\leb{\la+\mu-\nu\atop\nu-\bfj'',\,\la-\nu+\bfj'',\,
\mu-\nu}\rib;$$
\begin{equation*}
\begin{split}
(2)\quad \Sal(\bfl) A(\bfj,\la)
&=\sum_{T\in\afThn,\,\ro(T)=
\al\atop\bt,\eta\in\afmbnn,\,\bt\leq\dt_T,\,\bt+\eta\leq\la}
g_{\bt,\eta,T}\cdot
(A+T^\pm-\ti T^\pm)(\bfj_T+\la-\eta-2\bt,\dt_T+\eta),\qquad\quad\ \
\end{split}
\end{equation*}
where
$$g_{\bt,\eta,T}=\up^{f_{T}+(\eta+\bt)\centerdot(2\dt_T-\dt_{\ti T})}\leb{\dt_{\ti T}-\dt_T\atop\la-\eta-\bt}\rib\leb{\dt_T+\eta\atop\bt,\,\dt_T-\bt,\,\eta}\rib\prod_{1\leq i\leq n\atop j\not=i,\,j\in\mbz}\ol{\dleb{a_{i,j}+t_{i,j}-t_{i-1,j}\atop t_{i,j}}\drib}\in\sZ,$$
and $\bfj_T$, $f_T$ are defined as in Theorem~\ref{mul formulas in quantum affine gln}(2);
\begin{equation*}
\begin{split}
(3)\quad {}^t\!\Sal(\bfl) A(\bfj,\la)
&=\sum_{T\in\afThn,\,\ro(T)=
\al\atop\bt,\eta\in\afmbnn,\,\bt\leq\dt_{\ti T},\,\bt+\eta\leq\la}
g'_{\bt,\eta,T}\cdot
(A-T^\pm+\ti T^\pm)(\bfj_T'+\la-\eta-2\bt,\dt_{\ti T}+\eta),\qquad\quad\ \
\end{split}
\end{equation*}
where
$$g'_{\bt,\eta,T}=\up^{f'_{T}+(\eta+\bt)\centerdot(2\dt_{\ti T}-\dt_{T})}\leb{\dt_{T}-\dt_{\ti T}\atop\la-\eta-\bt}\rib\leb{\dt_{\ti T}+\eta\atop\bt,\,\dt_{\ti T}-\bt,\,\eta}\rib\prod_{1\leq i\leq n\atop j\not=i,\,j\in\mbz}\ol{\dleb{a_{i,j}-t_{i,j}+t_{i-1,j}\atop t_{i-1,j}}\drib}\in\sZ,$$
and $\bfj'_T$, $f_{T}'$ are defined as in Theorem~\ref{mul formulas in quantum affine gln}(3).
The same formulas hold in $\afSr$ with $A(\bfj,\la)$ etc. replaced by $A(\bfj,\la,r)$, etc.
\end{Prop}
\begin{proof} The fact $[A][B]\neq0\implies\ro(B)=\co(A)$ gives
$$0(\bfj',\mu,r)A(\bfj,\la,r)=\sum_{\al\in\La_\vtg(n,r-\sigma(A)}v^{(\ro(A)+\al)\centerdot\bfj'+\al\centerdot\bfj}\bigg[{\ro(A)+\al\atop\mu}\bigg]\bigg[{\al\atop\la}\bigg][A+\diag(\al)].$$
Applying \eqref{SGPidentity} yields the required formula. For more details, see \cite[3.4]{Fu1}.

Similarly, by Proposition~\ref{[B][A]}, the left hand side of (2) at level $r$ becomes
\begin{equation*}
\begin{split}
\Sal(\bfl,r) A(\bfj,\la,r) &=\sum_{\ga\in\afLa(n,r-\sg(A))}\up^{\ga\centerdot\bfj}\leb{\ga\atop\la}\rib
\left[S_\al+\diag\left(\ga+\ro(A)-\sum_{1\leq i\leq n}\al_i\afbse_{i+1}\right)\right][A+\diag(\ga)]\\
&=\sum_{T\in\afThn \atop \ro(T)=\al }
\prod_{1\leq i\leq n\atop j\in\mbz,\,j\not= i}\ol{\dleb{a_{i,j}+t_{i,j}-t_{i-1,j}\atop t_{i,j}}\drib}x_T
\end{split}
\end{equation*}
where
\begin{equation*}
\begin{split}
x_T&=
\sum_{\ga\in\afLa(n,r-\sg(A))}\up^{\ga\centerdot\bfj+\bt(T,A+\diag(\ga))}
\leb{\ga\atop\la}\rib
\ol{\dleb{\ga+\dt_T-\dt_{\ti T}\atop\dt_T}\drib}[A+T^\pm-\ti T^\pm+\diag(\ga+\dt_T-\dt_{\ti T})].
\end{split}
\end{equation*}
Let $\nu=\ga+\dt_T-\dt_{\ti T}$. Then
$\bt(T,A+\diag(\ga))=\bt_{A,T}+\bt_{\nu,T}$,
where $\bt_{\nu,T}=\sum_{1\leq i\leq n,\,i\geq l}\nu_it_{i,l}-\sum_{1\leq i\leq n,\,i+1>l}\nu_{i+1}t_{i,l}$ and
\begin{equation*}
\begin{split}
\bt_{A,T}&=\sum_{1\leq i\leq n\atop j\geq l,\,j\not=i}(a_{i,j}-t_{i-1,j})t_{i,l}-\sum_{1\leq i\leq n\atop j>l,\,j\not=i+1}a_{i+1,j}t_{i,l}+\sum_{1\leq i\leq n\atop j>l,\,j\not=i,i+1}t_{i,j}t_{i,l}\\
&\qquad -\sum_{1\leq i\leq n}t_{i,i}^2+\sum_{1\leq i\leq n\atop i+1>l}t_{i+1,i+1}t_{i,l}.
\end{split}
\end{equation*}
Furthermore, we have $\ol{\dbbl{\nu\atop\dt_T}\dbbr}=\up^{\dt_T\centerdot(\dt_T-\nu)}
\bbl{\nu\atop \dt_T}\bbr$,
$\bt_{A,T}+\dt_T\centerdot\dt_T+\bfj\centerdot(\dt_{\ti T}-\dt_T)=f_{T}$ and $\bt_{\nu,T}+\nu\centerdot(\bfj-\dt_T)=\nu\centerdot\bfj_T$.
This implies that
$$x_T=\sum_{\nu\in\afLa(n,r-\sg(A+T^\pm-\ti T^\pm))}\up^{f_{T}+\nu\centerdot\dt_{T,\bfj}}
\leb{\nu\atop\dt_T}\rib\cdot\leb{\nu-\dt_T+\dt_{\ti T}\atop\la}\rib
[A+T^\pm-\ti T^\pm+\diag(\nu)].$$
Applying the identities in \eqref{SGPidentity} yields
\begin{equation*}
\begin{split}
\leb{\nu\atop\dt_T}\rib\cdot\leb{\nu-\dt_T+\dt_{\ti T}\atop\la}\rib
&=\sum_{\bfx\in\afmbnn\atop\bfx\leq\la}\up^{\nu\centerdot(\la-\bfx)
-\bfx\centerdot(\dt_{\ti T}-\dt_T)}
\leb{\nu\atop\dt_T}\rib
\cdot\leb{\nu\atop\bfx}\rib\cdot\leb{\dt_{\ti T}-\dt_T\atop\la-\bfx}\rib\\
&=\sum_{\bfx,\bt\in\afmbnn,\,\bt\leq\dt_T
\atop\bt\leq\bfx\leq\la}\up^{\nu\centerdot(\la-\bfx-\bt)
+\bfx\centerdot(2\dt_T-\dt_{\ti T})}
\leb{\dt_{\ti T}-\dt_T\atop\la-\bfx}\rib
\leb{\dt_T+\bfx-\bt\atop\bt,\,\dt_T-\bt,\,\bfx-\bt}\rib \\
&\qquad\quad \times \leb{\nu\atop
\dt_T+\bfx-\bt}\rib.
\end{split}
\end{equation*}
Thus,
\begin{equation*}
\begin{split}
x_T&= \sum_{\bfx,\bt\in\afmbnn,\,\bt\leq\dt_T\atop
\bt\leq\bfx\leq\la}\up^{f_T+\bfx\centerdot(2\dt_T-\dt_{\ti T})}
\leb{\dt_{\ti T}-\dt_T\atop\la-\bfx}\rib
\leb{\dt_T+\bfx-\bt\atop\bt,\,\dt_T-\bt,\,\bfx-\bt}\rib\\
&\qquad\qquad\times(A+T^\pm-\ti T^\pm)(\bfj_T+\la-\bfx-\bt,\dt_T+\bfx-\bt)\\
&= \sum_{\eta,\bt\in\afmbnn,\,\bt\leq\dt_T\atop
\bt+\eta\leq\la}\up^{f_T+(\eta+\bt)\centerdot(2\dt_T-\dt_{\ti T})}
\leb{\dt_{\ti T}-\dt_T\atop\la-\eta-\bt}\rib
\leb{\dt_T+\eta\atop\bt,\,\dt_T-\bt,\,\eta}\rib\\
&\qquad\qquad\times(A+T^\pm-\ti T^\pm)(\bfj_T+\la-\eta-2\bt,\dt_T+\eta,r)
\end{split}
\end{equation*}
Consequently, (2) holds. Formula (3) can be proved similarly.
\end{proof}

\section{Lusztig form of $\bfU(\afgl)$ and integral affine quantum Schur--Weyl reciprocity}

We are now ready to determine the Lusztig form of $\bfU(\afgl)$ by proving the conjecture \cite[3.8.6]{DDF}.

Let $\afVn$ be the $\sZ$-submodule of $\afbfSn$ spanned by
$\{A(\bfj,\la)\mid A\in\afThnpm,\,\bfj\in\afmbzn,\,\la\in\afmbnn\}.$ As seen above, $\afVn$ is a $\sZ$-submodule of $\afbfVn$.
Our aim is to show that $\afVn$ is a realisation of $\dHa$ (see Theorem~\ref{realization of dHa} below). The following result is \cite[4.8]{Fu1}.

\begin{Lem}\label{basis for afVn}
The set $\{A(\bfj,\la)\mid A\in\afThnpm,\,
\bfj,\la\in\afmbnn,\,j_i\in\{0,1\},\forall i\}$ forms a $\sZ$-basis for $\afVn.$
\end{Lem}

\begin{proof}
Since the 0-part of $\bfU(\afgl)$ is the same as that of $\bfU(\mathfrak{gl}_n)$, the proof in the finite case
{\cite[4.2]{Fu1}} carries over.
\end{proof}

Let $\afVnp=\spann_\sZ\{A(\bfl)\mid A\in\afThnp\}$,
$\afVnm=\spann_\sZ\{A(\bfl)\mid A\in\afThnm\}$ and $\afVnz=\spann_\sZ\{0(\bfj,\la)\mid \bfj\in\afmbzn,\,\la\in\afmbnn\}$.
By Proposition~\ref{formula in Vz}(1),
$\afVnz$ is a $\sZ$-subalgebra of $\afbfSn$.

\begin{Lem}\label{Vzp-Vzm} The $\sZ$-module $\afVnp$ (resp., $\afVnm$) is a subalgebras of $\afbfSn$ which
 is generated by $(\sum_{1\leq i\leq n}\al_i\afE_{i,i+1})(\bfl)$ (resp., $(\sum_{1\leq i\leq n}\al_i\afE_{i+1,i})(\bfl)$) for $\al\in\afmbnn$ as a $\sZ$-algebra.
\end{Lem}
\begin{proof}

Since $\dHap=\Ha$ is a $\sZ$-subalgebra of $\dbfHa$ and $\tiafVnp=\zeta(\dHap)$ by Theorem~\ref{realization}, we conclude the first assertion which together with Proposition~\ref{tri Hall} gives the second assertion.
\end{proof}

We now recall the triangular relation for affine quantum Schur algebras.
For $A,B\in\aftiThn$ define
\begin{equation}\label{order sqsubset}
B\sqsubseteq A\text{ if and only if $B\pr A$, $\co(B)=\co(A)$ and $\ro(B)=\ro(A)$.}
\end{equation}
 Put $B\sqsubset A$ if $B\sqsubseteq A$ and $B\not=A$. According to \cite[6.1]{DF10}  the order relation $\sqsubseteq$ is a partial order relation on $\aftiThn$ with finite intervals $(-\infty,A]$ for all $A$; see Lemma~\ref{finite} below.

For $A\in\aftiThn$ with $\sg(A)=r$, we denote $[A]=0\in\afSr$ if $a_{i,i}<0$ for some $i\in\mbz$.
 For $A\in\aftiThn$ let $\bfsg(A)=(\sg_i(A))_{i\in\mbz}\in\afmbnn$ where $\sg_i(A)=a_{i,i}+\sum_{j<i}(a_{i,j}+a_{j,i})$.
The following triangular relation for affine quantum Schur algebras is given in \cite[3.7.7]{DDF}. The first assertion can be seen easily from the proof of loc. cit.
\begin{Prop}\label{tri-affine Schur algebras}
For $A\in\afThnpm$ and $\la\in\afLanr$, we have
\begin{equation*}
A^+(\bfl,r)[\diag(\la)]A^-(\bfl,r)
=[A+\diag(\la-\bfsg(A))]+\text{a $\sZ$-linear comb. of $[A']$ with $A'\sqsubset A$}.
\end{equation*}
In particular, the set $$\{A^+(\bfl,r)[\diag(\la)]A^-(\bfl,r)\mid A\in\afThnpm,\,\la\in\afLanr,\,\la\geq\bfsg(A)\}$$ forms a $\sZ$-basis for $\afSr$, where the order relation $\leq$ is defined in \eqref{order on afmbzn}.
\end{Prop}

For $w\in\ti\Sg$, let
\begin{equation*}
\begin{split}
\ttn_{(w)}^+&=\zeta(\ti u^+_{(w)})\in\afbfSn\quad\text{ and }\quad
\ttn_{(w)}^-=\zeta(\ti u^-_{(w)})\in\afbfSn.
\end{split}
\end{equation*}
The triangular relation for affine quantum Schur algebras can be lifted to the $\afbfSn$ level as follows.

\begin{Lem}\label{tri-Vz}
Let $A\in\afThnpm$, $\bfj\in\afmbzn$ and $\la\in\afmbnn$.

$(1)$ We have
$$A^+(\bfl) 0(\bfj,\la) A^-(\bfl)=
\sum_{\dt\in\afmbnn\atop\dt\leq\la}\up^{(\bfj-\dt)\centerdot\bfsg(A)}
\leb{\bfsg(A)\atop\la-\dt}\rib A(\bfj+\la-\dt,\dt)+f$$
where $f$ is a $\sZ$-linear combination of $B(\bfj',\dt)$ such that $B\in\afThnpm$, $B\p A$, $\dt\in\afmbnn$ and $\bfj'\in\afmbzn$. In particular,
We have $\afVn=\afVnp\afVnz\afVnm$.

$(2)$ There exist $w_{A^+},w_{A^-}\in\ti\Sg$ such that $\wp^+(w_{A^+})=A^+$, $\wp^-(w_{A^-}):={}^t\wp^+(w_{{}^t\!\!A^-})=A^-$ and
$$\ttn^+_{(w_{A^+})}0(\bfj,\la)\ttn^-_{(w_{A^-})}=
\sum_{\dt\in\afmbnn\atop\dt\leq\la}\up^{(\bfj-\dt)\centerdot\bfsg(A)}
\leb{\bfsg(A)\atop\la-\dt}\rib A(\bfj+\la-\dt,\dt)+g$$
where $g$ is a $\sZ$-linear combination of $B(\bfj',\dt)$ such that $B\in\afThnpm$, $B\p A$, $\dt\in\afmbnn$ and $\bfj'\in\afmbzn$.
\end{Lem}
\begin{proof}
According to Proposition~\ref{tri-affine Schur algebras}, for any $\mu\in\afLanr$, we have
\begin{equation*}
A^+(\bfl,r)[\diag(\mu)]A^-(\bfl,r)=[A+\diag(\mu-\bfsg(A))]+f_{\mu,r}
\end{equation*}
where $f_{\mu,r}$ is a $\sZ$-linear combination of $[B]$ such that $B\in\afThnr$ and  $B \sqsubset A+\diag(\mu-\bfsg(A))$. Thus,
\begin{equation*}
\begin{split}
A^+(\bfl,r)0(\bfj,\la,r) A^-(\bfl,r)&=\sum_{\mu\in\afLanr}
\up^{\bfj\centerdot\mu}\leb{\mu\atop\la}\rib
([A+\diag(\mu-\bfsg(A))]+f_{\mu,r})\\
&=\sum_{\nu\in\afLa(n-r-\sg(A))}\up^{\bfj\centerdot(\nu+\bfsg(A))}
\leb{\nu+\bfsg(A)\atop\la}\rib
[A+\diag(\nu)]+f_r,\\
\end{split}
\end{equation*}
where $f_r=\sum_{\mu\in\afLanr}\up^{\bfj\centerdot\mu}
\leb{\mu\atop\la}\rib f_{\mu,r}$. By \eqref{SGPidentity}, we have
\begin{equation*}
\begin{split}
A^+(\bfl,r)0(\bfj,\la,r) A^-(\bfl,r)
&=\sum_{\nu\in\afmbzn}\up^{\bfj\centerdot(\nu+\bfsg(A))}
\sum_{\dt\in\afmbnn\atop\dt\leq\la}
\up^{\nu\centerdot(\la-\dt)-\dt\centerdot\bfsg(A)}
\leb{\nu\atop\dt}\rib
\leb{\bfsg(A)\atop\la-\dt}\rib
[A+\diag(\nu)]+f_r\\
&=\sum_{\dt\in\afmbnn\atop\dt\leq\la}\up^{(\bfj-\dt)\centerdot\bfsg(A)}
\leb{\bfsg(A)\atop\la-\dt}\rib A(\bfj+\la-\dt,\dt)+f_r.
\end{split}
\end{equation*}
 On the other hand, by Lemma~\ref{Vzp-Vzm} and  Proposition~\ref{formula in Vz}, we see that $(f_r)_{r\geq 0}\in\afVn$. Hence, $(f_r)_{r\geq 0}$ must be a $\sZ$-linear combination of $B(\bfj',\dt)$ such that $B\in\afThnpm$, $B\p A$, $\dt\in\afmbnn$ and $\bfj'\in\afmbzn$. This proves (1). The assertion (2) follows from (1), Proposition~\ref{tri Hall} and Theorem~\ref{zr}.
\end{proof}

For $A\in\aftiThn$, let
$$\ddet{A}=\sum_{i<j\atop1\leq i\leq n}{j-i+1\choose 2}(a_{i,j}+a_{j,i}).$$
Then, $A\prec B$ implies $\ddet{A}<\ddet{B}$.
The following result is the affine version of \cite[Prop.~4.3]{Fu1} which is conjectured in \cite[4.9]{Fu1}.

\begin{Prop}\label{Vz-subalgebra}
The $\sZ$-module $\afVn$ is a subalgebra of $\afbfSn$ which is
generated by the elements $(\sum_{1\leq i\leq n}\al_i\afE_{i,i+1})(\bfl)$, $(\sum_{1\leq i\leq n}\al_i\afE_{i+1,i})(\bfl)$, $0({\afbse_i})$, $0(\bfl,t\afbse_i)$ for all $\al\in\afmbnn$, $t\in\mbn$, $1\leq i\leq n$.
\end{Prop}
\begin{proof}
Let $\afVn'$ be the $\sZ$-subalgebra of $\afbfSn$ generated by the indicated elements. According to Proposition~\ref{formula in Vz}, we have $\afVn'\han\afVn'\afVn\han\afVn$.
We shall show by induction on $\ddet A$ that $A(\bfj,\la)\in\afVn'$ for all $A\in\afThnpm$, $\bfj\in\afmbzn$ and $\la\in\afmbnn$. If $\ddet A=0$, then $A=0$ and $0(\bfj,\la)=\prod_{1\leq i\leq n}0(\afbse_i)^{j_i}0(\bfl,\la_i\afbse_i)\in\afVn'.$
Now we assume that $\ddet A>0$ and $A'(\bfj,\la)\in\afVn'$
for all $A',\bfj,\la$ with $\ddet {A'}<\ddet A$.
By Lemma~\ref{tri-Vz}(2) and \cite[3.7.6]{DDF}, there exist $w_{A^+},w_{A^-}\in\ti\Sg$ such that
$$\ttn^+_{(w_{A^+})}\ttn^-_{(w_{A^-})}=A\br\bfl+g$$
where $g$ is a $\sZ$-linear combination of $B(\bfj',\dt)$ with $B\in\afThnpm$, $|\!|B|\!|<|\!|A|\!|$, $\dt\in\afmbnn$ and $\bfj'\in\afmbzn$.
By the induction hypothesis we have $g\in\afVn'$. It follows that $A(\bfl)\in\afVn'$ and so $A(\bfj)\in\afVn'$ by
Theorem~\ref{mul formulas in quantum affine gln}(1). Furthermore,  by Proposition~\ref{formula in Vz}(1) (setting $\bfj'=\mu-\nu$ there),
\begin{equation}\label{0{la}A{0}}
\begin{split}
0(\bfj,\la) A(\bfl)&=
\up^{\ro(A)\centerdot(\bfj+\la)}A(\bfj,\la)+
\sum_{\bfj'\in\afmbnn\atop\bfj'<\la}\up^{\ro(A)\centerdot(\bfj+\bfj')}
\leb{\ro(A)\atop\la-\bfj'}\rib A(\bfj+\bfj'-\la,\bfj')\\
&=\up^{\ro(A)\centerdot(\bfj+\la)}A(\bfj,\la)+
\sum_{\bfj'\in\afmbnn\atop\sg(\bfj')<\sg(\la)}\up^{\ro(A)\centerdot(\bfj+\bfj')}
\leb{\ro(A)\atop\la-\bfj'}\rib A(\bfj+\bfj'-\la,\bfj').
\end{split}
\end{equation}
Thus, by induction on $\sg(\la)$, we conclude that $A(\bfj,\la)\in\afVn'$ for all $\bfj\in\afmbzn$ and $\la\in\afmbnn$. \end{proof}

As indicated in \cite[Rem. 4.10(3)]{Fu1}, we  now  use Proposition~\ref{Vz-subalgebra} to prove the conjecture formulated in \cite[3.8.6]{DDF}.
Recall from Theorem~\ref{realization} that the homomorphism $\zeta$ in \eqref{zeta} induces an isomorphism $\zeta:\dbfHa\to\afbfVn$.

\begin{Thm}\label{realization of dHa} We have $\zeta^{-1}(\afVn)=\dHa$. In particular,
$\dHa$ is a subalgebra of $\dbfHa$ isomorphic to $\afVn$. Moreover, $\dHa$ is a Hopf subalgebra of $\dbfHa$.
\end{Thm}
\begin{proof}Since
$\zeta(\dHa)=\zeta(\dHap)\zeta(\dHaz)\zeta(\dHam)=\afVnp\afVnz\afVnm,$
it follows from Lemma~\ref{tri-Vz}(1) that $\zeta(\dHa)=\afVn$.
Hence, by Proposition~\ref{Vz-subalgebra} and Theorem~\ref{realization}, $\dHa$ is a subalgebra. By using the semisimple generators for $\dHa$, the last assertion follows from \cite[3.5.7]{DDF}.
\end{proof}

\begin{Rem}
(1) A different integral form $U_\up^\res(\afgl)$ of $\bfU(\afgl)$ was constructed
in \cite[7.2]{FM}. As pointed
out in \cite{FM}, it is not known if $U_\up^\res(\afgl)$ is a Hopf
subalgebra. It would be interstring to find a relation
between $\dHa$ and $U_\up^\res(\afgl)$.

(2) There is another form using the Lusztig form of $\bfU(\afsl)$ tensoring with an integral central algebra;
see \cite[2.4.4]{DDF}. However, this form does not map onto the integral affine quantum Schur algebras; see Example 5.3.8 in \cite{DDF}.
\end{Rem}

We end this section with an application to the affine quantum Schur--Weyl reciprocity at the integerl level. The proof of the following result is the same as that of \cite[Th.~3.8.1(1)]{DDF}.

\begin{Thm}\label{affine Schur-Weyl duality}
The restriction of $\zr$ to $\dHa$ gives a surjective $\sZ$-algebra homomorphism
$$\zr:\dHa\twoheadrightarrow\afSr.$$
\end{Thm}

Let $\mpk$ be a commutative ring containing an invertible element $\vep$. We will regard $\mpk$ as a  $\sZ$-module by specializing
$\up$ to $\vep$.
Let  $\dHak=\dHa\ot_\sZ \mpk$,
$\afSrk=\afSr\ot_\sZ \mpk$. Then we have $\afSrk\cong\End_{\afHrk}(\afTnrk),$
where $\afTnrk=\oplus_{\la\in\afLanr}(x_\la\sH_\vtg(r)_\mpk)$ with $\afHrk=\afHr\ot_\sZ\mpk$.
\begin{Coro}
For any commutative ring $\mpk$, there is an algebra epimorphism
$$\zr\ot 1:\dHak\twoheadrightarrow\afSrk.$$
\end{Coro}

\section{The affine BLM algebra $\afKn$}
We first derive in Proposition~\ref{stabilization property} the affine stabilisation property for affine quantum Schur algebras, which is the affine analogue of \cite[4.2]{BLM}. We then construct the affine BLM algebra $\afbfKn$ and prove that it is isomorphic to the modified quantum group $\ddbfHa$.

Observe the structure constants in Proposition~\ref{[B][A]} and separate the Gaussian polynomial $\dblr{a_{i,i}+t_{i,i}-t_{i-1,i}\atop t_{i,i}}$ from the product. We now introduce, for
 a second indeterminate $\up'$, $T\in\afThn$ and $A\in\aftiThn$, the polynomials
$$P_{T,A}(\up,\up')=
\up^{\bt(T,A)}\prod_{1\leq i\leq n\atop j\in\mbz,\,j\not=i}\ol{\dleb{{a_{i,j}+t_{i,j}-t_{i-1,j}\atop t_{i,j}}}\drib}\prod_{1\leq i\leq n\atop 1\leq s\leq t_{i,i}}\frac{\up^{-2(a_{i,i}+t_{i,i}-t_{i-1,i}-s+1)}\up'^2-1}{\up^{-2s}-1}$$
and
$$Q_{T,A}(\up,\up')=\up^{\bt'(T,A)}\prod_{1\leq i\leq n\atop j\in\mbz,\,j\not=i}\ol{\dleb{{a_{i,j}-t_{i,j}+t_{i-1,j}\atop t_{i,j}}}\drib}\prod_{1\leq i\leq n\atop 1\leq s\leq t_{i-1,i}}\frac{\up^{-2(a_{i,i}-t_{i,i}+t_{i-1,i}-s+1)}\up'^2-1}{\up^{-2s}-1}$$
in the subring $\sZ_1$ of $\mbq(\up)[\up',\up^{\prime-1}]$, where
\begin{equation}\label{sZ_1}
\text{$\sZ_1$ is generated (over $\mbz$!) by
$\prod_{1\leq i\leq t}\frac{\up^{-2(a-i)}\up'^2-1} {\up^{-2i}-1}$, $\prod_{1\leq i\leq t}\frac{\up^{2(a-i)}\up'^{-2}-1} {\up^{2i}-1}$,
and $\up^j$}
\end{equation}
for all $a\in\mbz$, $t\geq 1$ and $j\in\mbz$. Note that $\sZ_1|_{v'=1}=\sZ$.

For $A\in\aftiThn$ and $p\in\mbz$, let
$${}_pA=A+pI$$
where $I\in\afThn$ is the identity matrix. Then it is clear that
$\beta(T,A)=\beta(T,{}_pA)$ and $\beta'(T,A)=\beta'(T,{}_pA)$.
Thus, Proposition~\ref{[B][A]} can be generalised as follows.

\begin{Lem}\label{stabilization property1}
Let $A,B\in\aftiThn$ and assume $\co(B)=\ro(A)$ and $b=\sg(A)=\sg(B)$.

$(1)$ If $B-\sum_{1\leq i\leq n}\al_i\afE_{i,i+1}$ is diagonal for some $\al\in\afmbnn$ then, for large $p$ and $r=pn+b$, we have in $\afSr$:
$$[{}_pB][{}_pA]=\sum_{T\in\afThn,\,\ro(T)=\al\atop a_{i,j}+t_{i,j}-t_{i-1,j}\geq 0,\,\forall i\not=j}P_{T,A}(\up,\up^{-p})[{}_p(A+T-\ti T)].$$

$(2)$ If $B-\sum_{1\leq i\leq n}\al_i\afE_{i+1,i}$ is diagonal for some $\al\in\afmbnn$ then, for large $p$ and $r=pn+b$, we have in $\afSr$:
$$[{}_pB][{}_pA]=\sum_{T\in\afThn,\,\ro(T)=\al\atop a_{i,j}-t_{i,j}+t_{i-1,j}\geq 0,\,\forall i\not=j}Q_{T,A}(\up,\up^{-p})[{}_p(A-T+\ti T)].$$
\end{Lem}

Let
$\aftiThn^{ss}$ be the set of $X\in\aftiThn$ such that either $X-\sum_{1\leq i\leq n}\al_i\afE_{i,i+1}$ or $X-\sum_{1\leq i\leq n}\al_i\afE_{i+1,i}$ is diagonal for some $\al\in\afmbnn$.
We have the following affine version of \cite[3.9]{BLM} (see \cite[4.5]{Fu} for a slightly different version). For completeness, we include a proof.
\begin{Prop}\label{generalization of BLM 3.9}
Let $A\in\afTh(n,r)$. Then there exist upper triangular matrices  $A_1,A_2,\cdots,A_s$ and lower triangular matrices $A_{s+1},A_{s+2},\cdots,A_{t}$ in $\aftiThn^{ss}\cap \afTh(n,r)$ such that $\co(A_i)=\ro(A_{i+1})$ ($1\leq i\leq t-1$) and the following identity holds in $\afSpr$: for $p\geq0$,
\begin{equation*}
[{}_p(A_1)]
\cdots
[{}_p(A_s)]\cdot[{}_p(A_{s+1})]
\cdots
[{}_p(A_{t})]=[{}_pA]+\text{lower terms relative to }\sqsubset.
\end{equation*}
\end{Prop}

\begin{proof} By Proposition~\ref{tri Hall}, there is a distinguished words $w_B$ for every $B\in\afThnp$ satisfying the triangular relation \eqref{eq tri Hall}. Let $x=w_{A^+}$ and $y={}^tw_{{}^t\!A^-}$. By Theorem~\ref{zr} and Proposition~\ref{tri Hall}, we have in $\afSr$
$$
\ttm_{(x),r}^+:=\zeta_r(\ti u^+_{(x)})=A^+(\bfl,r)+f\quad\text{ and }\quad
\ttm_{(y), r}^-:=\zeta_r(\ti u^-_{(y)})=A^-(\bfl,r)+g,$$
where $f$ (resp., $g$) is a linear combination of $B(\bfl,r)$  with $B\in\afThnp$ (resp., $B\in\afThnm$) and $B\prec A^+$ (resp., $B\prec A^-$). By Proposition~\ref{tri-affine Schur algebras}, we have for $p\geq0$
 $$\ttm_{(x),r}^+[\diag(\boldsymbol\sigma({}_pA))]\ttm_{(y),r}^-=[{}_pA]+\text{lower terms}.$$
Finally, by writing the words $x,y$ in full, it is clear to see that there exist upper triangular matrices  $A_1,A_2,\cdots,A_s$ and lower triangular matrices $A_{s+1},A_{s+2},\cdots,A_{t}$ in $\aftiThn^{ss}$ such that
$$\ttm_{(x),r}^+[\diag(\boldsymbol\sigma({}_pA))]=[{}_p(A_1)]\cdots[{}_p(A_s)]\quad\text{ and }\quad
[\diag(\boldsymbol\sigma({}_pA))]\ttm_{(y),r}^-=[{}_p(A_{s+1})]
\cdots
[{}_p(A_{t})],$$
as desired.
\end{proof}

We can now prove the following stabilization property for affine quantum Schur algebras.

\begin{Prop}\label{stabilization property}
Let $A,B\in\aftiThn$ and assume $\co(B)=\ro(A)$. Then there exist unique $X_1,\cdots,X_m\in\aftiThn$, unique $P_1(\up,\up'),\cdots,P_m(\up,\up')\in\sZ_1$ and an integer $p_0\geq 0$ such that, in $\afSpA$,
\begin{equation}\label{eq stabilization}
[{}_pB][{}_pA]=\sum_{1\leq i\leq m}P_i(\up,\up^{-p})[{}_pX_i]\quad\text{for all $p\geq p_0$}.
\end{equation}

\end{Prop}
\begin{proof} The proof can be conducted by induction on $\ddet{B}$. With Lemma~\ref{stabilization property1} and Proposition~\ref{generalization of BLM 3.9}, the proof is entirely similar to that of \cite[3.9]{BLM} or \cite[Prop. 14.1]{DDPW}.
\end{proof}

Let $\aftisK$ be the free $\sZ_1$-module with basis $\{A\mid A\in\aftiThn\}$. Then, by Proposition~\ref{stabilization property}, we may make $\aftisK$ into an associative $\sZ_1$-algebra (without unit) by the multiplication:
\begin{equation}\label{tiKn}
B\cdot A=\begin{cases}\sum_{1\leq i\leq m}P_i(\up,\up')X_i, &\text{  if $\co(B)=\ro(A)$};\\
0,&\text{ otherwise.}\end{cases}
\end{equation}
Let $$\afsK=\aftisK\ot_{\sZ_1}\sZ,$$
where $\sZ$ is regarded as a $\sZ_1$-module by specializing $\up'$ to $1$. Then $\afsK$ becomes an associative $\sZ$-algebra with basis $\{\dob A\dcb:=A\ot 1\mid A\in\aftiThn\}$.
Let $\afbfKn=\afsK\ot_\sZ\mbq(\up)$. 

Following \cite[5.1]{BLM}, let $\afhbfKn$   be
the vector space of all formal (possibly infinite) $\mbq(\up)$-linear combinations
$\sum_{A\in\aftiThn}\beta_A\dob A\dcb$ such that,
for any ${\bf x}\in\mathbb Z^n$, the sets
${\{A\in\aftiThn\ |\ \beta_A\neq0,\ \ro(A)={\bf
x}\}}$ and ${\{A\in\aftiThn\ |\ \beta_A\neq0,\ \co(A)={\bf
x}\}}$ are finite.
We can define the product of two elements
$\sum_{A\in\widetilde\Xi}\beta_A\dob A\dcb$,
$\sum_{B\in\widetilde\Xi}\gamma_B\dob B\dcb$ in $\afhbfKn$ to be
$\sum_{A,B}\beta_A\gamma_B\dob A\dcb\dob B\dcb$.
This defines an associative algebra structure on $\afhbfKn$. The algebra $\afbfVn$ can also be realized as a $\mbq(\up)$-subalgebra of $\afhbfKn$, which we now describe.

The following result can be proved in a way similar to the proof of \cite[6.7]{DF} (cf. \cite[6.3]{Fu}).

\begin{Lem}\label{dzr}
The linear map $\dzr:\afsK\ra\afSr$ defined by
\begin{equation}\label{dzr([A])}
\dzr(\dob A\dcb)=\begin{cases}[A]& \mathrm{if\ }A\in\afThnr;\\
0&  \mathrm{otherwise}\end{cases}
\end{equation}
is an algebra epimorphism.
\end{Lem}

The map $\dzr:\afsK\ra\afSr$  induces a surjective algebra homomorphism
\begin{equation}\label{hzr}
\hzr:\afhbfKn\ra\afbfSr
\end{equation}
sending $\sum_{A\in\aftiThn}\bt_A\dob A\dcb$ to $\sum_{A\in\aftiThn}\bt_A\dxr(\dob A\dcb)$.
Consequently, we get a surjective algebra homomorphism
\begin{equation}\label{hz}
\hz:\afhbfKn\twoheadrightarrow\afbfSn.
\end{equation}
defined by sending $x$ to $\hz(x):=(\hzr(x))_{r\geq 0}$. It is clear that we have $\hz(\afbfKn)=\afbfSno$ where $\afbfSno=\bop_{r\geq 0}\afbfSr$. Thus, by restriction $\hz$ to $\afbfKn$, we get a surjective algebra homomorphism from
$\afbfKn$ to $\afbfSno$.

For $A\in\afThnpm$, $\bfj\in\afmbzn$ and $\la\in\afmbnn$, let
$$A\dop\bfj\dcp:=\sum_{\mu\in\afmbzn}\up^{\mu\centerdot\bfj}
\dob A+\diag(\mu)\dcb\;\text{ and }\;A\dop\bfj,\la\dcp:=\sum_{\mu\in\afmbzn}\up^{\mu\centerdot\bfj}
\leb{\mu\atop\la}\rib\dob A+\diag(\mu)\dcb.$$
By Proposition~\ref{tri-affine Schur algebras}, the stabilisation property Proposition~\ref{stabilization property}
  implies that for any
$A\in\aftiThn$,
\begin{equation}\label{dot tri relation}
A^+\dop\bfl\dcp\dob\diag(\bfsg(A))\dcb A^-\dop\bfl\dcp
=\dob A\dcb+\text{a $\sZ$-linear comb. of $\dob A'\dcb$ with $A'\sqsubset A$}.
\end{equation}

Let $\afbfsfVn$ be the $\mbq(\up)$-subspace of $\afhbfKn$ spanned by all $A\dop\bfj\dcp$ ($A\in\afThnpm$ and $\bfj\in\afmbzn$). Let $\afsfVn$ be the $\sZ$-submodule of $\afhbfKn$ spanned by $A\dop\bfj,\la\dcp$ for all $A,\bfj,\la$ as above.

\begin{Thm}\label{unmodified}
(1) $\afbfsfVn$ is a subalgebra of $\afhbfKn$ and
the restriction of $\hz$ to $\afbfsfVn$ induces an  algebra isomorphism $\hz:\afbfsfVn\ra\afbfVn, A\dop\bfj\dcp\mapsto A(\bfj)$.

(2) The $\sZ$-module $\afsfVn$ is a subalgebra of $\afhbfKn$ and
the restriction of $\hz$ to $\afsfVn$ induces an  algebra isomorphism $\hz:\afsfVn\ra\afVn, A\dop\bfj,\la\dcp\mapsto A(\bfj,\la)$.
\end{Thm}
\begin{proof}By looking at the kernel of $\hz$ (cf. \cite[\S8]{DF10}), it is clear that the restriction of $\hz$ to ${\afbfsfVn}$ is injective. Note that $\hz(\afbfsfVn)=\afbfVn$ and $\hz(\afsfVn)=\afVn$. Now the assertion follows from
Theorem~\ref{realization} and Proposition~\ref{Vz-subalgebra}.
\end{proof}
This result together with Theorem~\ref{realization} gives another realisation of $\bfU(\afgl)$. This is an {\it unmodified} affine generalisation of the BLM construction in \cite{BLM}. {\it In particular, we will identify $\dbfHa$ with $\afbfsfVn$ and
$\dHa$ with $\afsfVn$ in the sequel.}

We end this section with a discussion on a realisation of the modified quantum group $\ddbfHa$. We will prove that $\ddbfHa$ and its integral form $\ddHa$ is isomorphic the affine BLM algebras $\afbfKn$ and $\afKn$, respectively.

 Let  $\afPin=\{\afbse_j-\afbse_{j+1}\mid 1\leq j\leq n\}.$
According to \cite[3.5.2]{Fu13}, the algebra $\dbfHa$ is a $\afmbzn$-graded algebra with
$\deg(u_A^+)=\ro(A)-\co(A),\
\deg(u_A^-)=\co(A)-\ro(A)\ \text{and}\ \deg(K_i^{\pm 1})=0$
for $A\in\afThnp$ and $1\leq i\leq n$.
For $\nu\in\afmbzn$, let $\dbfHa_\nu$ be the set of homogeneous elements in $\dbfHa$ of degree $\nu$. Then we have
$\dbfHa=\bop_{\nu\in\mbz\afPin}\dbfHa_\nu$.

For $\la,\mu\in\afmbzn$ we set ${}_\la\dbfHa_\mu=\dbfHa/{}_\la I_\mu$, where
\begin{equation}\label{laImu}
{}_\la I_\mu=\big(\sum_{\bfj\in\afmbzn}(\Kbfj-
 \up^{\la\cdot\bfj})\dbfHa+\sum_{\bfj\in\afmbzn}\dbfHa(\Kbfj
 -\up^{\mu\cdot\bfj})\big).
 \end{equation}
Let $\pi_{\la,\mu}:\dbfHa\ra{}_\la\dbfHa_\mu$ be the canonical projection. Since $\pi_{\la,\mu}(\dbfHa_{\la-\mu})={}_\la\dbfHa_\mu$ (cf. \cite[Lemma 6.2]{DF}), it follows that ${}_\la\dbfHa_\mu$ is spanned by the elements
$\pi_{\la,\mu}(u_A^+u_B^-)$   for all $A,B,\la,\mu$ with
$\la-\mu=\deg(u_A^+u_B^-)$.
Let
$$\ddbfHa:=\bop_{\la,\mu\in\afmbzn}{}_\la\dbfHa_\mu.$$

We define the product in $\ddbfHa$ as follows.
For $\la',\mu',\la'',\mu''\in\afmbzn$ with
$\la'-\mu',\la''-\mu''\in\mbz\afPin$ and any $t\in\dbfHa_{\la'-\mu'}$,
$s\in\dbfHa_{\la''-\mu''}$,  the product
$\pi_{\la',\mu'}(t)\pi_{\la'',\mu''}(s)$ is equal to $\pi_{\la',\mu''}(ts)$ if  $\mu'=\la''$, and it is zero, otherwise.
Then $\ddbfHa$ becomes an associative $\mbq(\up)$-algebra with this product. The algebra $\ddbfHa$ is naturally a $\dbfHa$-bimodule defined by
$t'\pi_{\la',\la''}(s)t''=\pi_{\la'+\nu',\la''-\nu''}(t'st'')$,
for $t'\in\dbfHa_{\nu'}$, $s\in\dbfHa$, $t''\in\dbfHa_{\nu''}$ and $\la',\la''\in\afmbzn$ (cf. \cite{Lubk,Fu13}).
In particular, putting $1_\la=\pi_{\la,\la}(1)$, we have $u_A^+1_\la
u_B^-=\pi_{\la+\deg(u_A^+),\la-\deg(u_B^-)}(u_A^+u_B^-)$ and $\ddbfHa$ is spanned by the elements $u_A^+1_\la u_B^-$ for all $A,B,\la$.

Let $\ddHa$ be the $\sZ$-submodule of $\ddbfHa$ spanned by the elements $u_A^+1_\la u_B^-$ for $A,B\in\afThnp$ and $\la\in \afmbzn$. It is proved in \cite[Th.~4.2]{Fu13} that $\ddHa$ is a $\sZ$-subalgebra of $\ddbfHa$. We now can realise $\ddbfHa$ and $\ddHa$ as $\afbfKn$ and $\afKn$, respectively; cf. \cite[Th.~6.3]{DF}.

\begin{Thm}\label{realization of ddHa}
The linear map $\Phi:\ddbfHa\ra\afbfKn$ sending $\pi_{\la\mu}(u)$ to
$\dob\diag(\la)\dcb u\dob\diag(\mu)\dcb$ for all $u\in\dbfHa$ and
$\la,\mu\in\afmbzn$, is an algebra isomorphism. Furthermore we have $\Phi(\ddHa)=\afKn$.
\end{Thm}
\begin{proof}By a proof similar to that of \cite[6.3]{DF},
it is easy to see that $\Phi$ is an algebra homomorphism. In particular, $\Phi(1_\la)=\dob\diag(\la)\dcb$.
By \eqref{dot tri relation}, the image of the spanning set $\{u_A^+1_\la u_B^-\mid A,B\in\afThnp,\la\in \afmbzn\}$
is in fact a basis for $\afbfKn$, proving the first assertion which implies the last assertion by definition.
\end{proof}
We will identify $\ddbfHa$ with $\afbfKn$ and $\ddHa$ with $\afKn$ via the map $\Phi$ defined in Theorem~\ref{realization of ddHa} and identify $\dbfHa$ with $\afbfsfVn$ and $\dHa$ with $\afsfVn$ as in Theorem~\ref{unmodified}. Then the $\dbfHa$-bimodule structure on $\ddbfHa$ satisfies the following simple formula: for all $A\in\afThnpm,\bfj,\la\in\afmbzn$,
\begin{equation}\label{bimodule}
A\dop\bfj\dcp\dob \diag(\la)\dcb=\dob A+\diag(\la-\co(A))\dcb,\quad \dob \diag(\la)\dcb A\dop\bfj\dcp=\dob A+\diag(\la-\ro(A))\dcb.
\end{equation}

For $A\in\aftiThn$, choose words $w_{A^+},w_{A^-}\in\ti\Sg$ such
that  \eqref{eq tri Hall} and its opposite version (obtained by applying \eqref{tau} to \eqref{eq tri Hall}) hold. Then, by
\eqref{dot tri relation},
\begin{equation}\label{tri afKn}
\mpM^{(A)}:=\ti u^+_{(w_{A^+})}1_{\bfsg(A)}\ti u_{(w_{A^-})}^-=\ti u^+_{(w_{A^+})}\dob\diag(\bfsg(A))\dcb\ti u_{(w_{A^-})}^-=\dob A\dcb+\sum_{B\sqsubset A\atop B\in\aftiThn}h_{A,B}\dob B\dcb,
\end{equation}
where $h_{A,B}\in\sZ$. Thus, we have immediately:

\begin{Coro}\label{ddHa monomial basis} The set $\{\mpM^{(A)}\mid A\in\aftiThn\}$ forms a $\sZ$-basis
for $\ddHa$.
\end{Coro}

\

\section{Canonical bases for the integral modified quantum affine $\frak{gl}_n$}

It is well known that the positive part of a quantum enveloping algebra $\bfU$ has a canonical basis with remarkable properties (see \cite{Kas1}, [23], [24]). In contrast, there is no canonical basis for $\bfU$. However, the modified form $\dot\bfU$ of $\bfU$ can have a canonical basis (see \cite{Kas2}, \cite{Lu92}, \cite{Lubk}).
We now define the canonical basis relative the basis $\{\dob A\dcb\}_{A\in\aftiThn}$ for $\ddHa=\afsK$. Our strategy is to use a stabilisation property for the bar involution on $\afSr$ to define a bar involution on $\aftisK$
(see \eqref{tiKn}) which then induces a bar involution on $\afsK$.

We first define the bar involution on $\afSr$ via the one on the Hecke algebra, following \cite{Du92} (cf. \cite{VV99}).
Let $W_r$ be the subgroup of $\afsygr$ generated by $s_i$ for $1\leq i\leq r$.
Let $\rho$ be the permutation of $\mbz$ sending $j$ to $j+1$ for all $j\in\mbz$.
Let $\sH(W_r)$ be the $\sZ$-subalgebra of $\afHr$ generated by $T_{s_i}$ for $1\leq i\leq r$.
Let $\{C_w'\mid w\in W_r\}$ be the canonical basis of $\sH(W_r)$ defined in \cite[1.1(c)]{KL79}. For $w=\rho^ax\in\affSr$ with $a\in\mbz$ and $x\in W_r$, let $C'_w=T_\rho^aC_x'.$
Then the set $\{C'_w\mid w\in\affSr\}$ forms a $\sZ$-basis for $\afHr$.
Note that $C_{w_0,\mu}'=v^{-\ell(w_{0,\mu})}x_\mu$. Let $\bar\ :\afHr\ra\afHr$ be the ring involution
defined by $\bar v=v^{-1}$ and $\bar T_w=T_{w^{-1}}^{-1}$.
We define a map $\bar\ :\afSr\ra\afSr$ such that
$\bar v=v^{-1}$ and $\bar f(C_{w_0,\mu}'h)=\ol{f(C_{w_0,\mu}')}h$ for $f\in\Hom_{\afHr}(x_{\mu}\afHr,x_\la\afHr)$ and $h\in\afHr$. Then the map $\bar\ :\afSr\ra\afSr$ is a ring involution.\footnote{See \cite[Prop. 3.2]{Du92} for a proof.}  We need to look some first properties of the bar involution in Lemma~\ref{bar monomial} before proving its stabilisation property in Proposition~\ref{bar stab}.

Given $A\in\afThnr$, write $y_A=w$ if  $A=\jmath_\vtg(\la,w,\mu)$, and also write $y_A^+$ for the  unique longest element in $\fS_\la w\fS_\mu$. For $\la\in\afLanr$, let $w_{0,\la}$ be the longest element in $\fS_\la$.

\begin{Lem}\label{longest element}
For $A\in\afThnr$ we have $\ell(y_A^+)=d_A+\ell(w_{0,\mu})$ where $\mu=\co(A)$ and $d_A$ is given in \eqref{dA}.
\end{Lem}
\begin{proof}
For $1\leq i\leq n$, let $\nu^{(i)}$ be the composition of $\mu_i$ obtained by removing all zeros from column $i$ of $A$. Let $\la=\ro(A)$.
According to \cite[3.2.3]{DDF},
$y_A^{-1}\fS_\la y_A\cap\fS_\mu=\fS_\nu$, where $\nu=(\nu^{(1)},\cdots,\nu^{(n)})$. Let $x$ be the longest element in $\afmsD_\nu\cap\fS_\mu$. Then $y_A^+=w_{0,\la}y_A x$ and  $\ell(y_A^+)=\ell(w_{0,\la})+\ell(y_A)+\ell(x)$.
Since $w_{0,\nu}x$ is the longest element in $\fS_\mu$, it follows that $w_{0,\mu}=w_{0,\nu}x$ and
$$\ell(x)=\ell(w_{0,\mu})-\ell(w_{0,\nu})=\sum_{1\leq i\leq n}\bigg(\bigg({\mu_i\atop 2}\bigg)-\sum_{k\in\mbz}\bigg({\nu_k^{(i)}\atop 2}\bigg)\bigg)
=\sum_{1\leq i\leq n\atop s<t}\nu_s^{(i)}\nu_t^{(i)}.$$
Hence,
\begin{equation}\label{l yA+}
\ell(y_A^+)=\ell(w_{0,\la})+\ell(y_A)+\ell(x)=
\ell(w_{0,\la})+\ell(y_A)+\sum_{1\leq i\leq n\atop s<t}a_{s,i}a_{t,i}
\end{equation}
By \cite[5.3]{DF13},
$\label{dA-l(wA)}
d_A-\ell(y_A)=\sum_{1\leq i\leq n;\,
j<l}a_{i,j}a_{i,l}.$
Furthermore, we have
$$
\ell(w_{0,\la})-\ell(w_{0,\mu})=\sum_{1\leq i\leq n}\bigg(\frac{\la_i(\la_i-1)}{2}-\frac{\mu_i(\mu_i-1)}{2}\bigg)
=\sum_{1\leq i\leq n\atop k<l}(a_{i,k}a_{i,l}-a_{k,i}a_{l,i}).
$$
Thus, by \eqref{l yA+}, we conclude that
$d_A-\ell(y_A)-(\ell(w_{0,\la})-\ell(w_{0,\mu}))=\ell(y_A^+)-\ell(w_{0,\la})-\ell(y_A).$
Consequently, $\ell(y_A^+)=d_A+\ell(w_{0,\mu})$.
\end{proof}

For $d\in\afmsD_{\la,\mu}$ let
$$\ti T_{\fS_\la d\fS_\mu}=v^{-\ell(d^+)}T_{\fS_\la d\fS_\mu},$$
where $d^+$ is the unique longest element in $\fS_\la d\fS_\mu$. Recall from Theorem~\ref{mul formulas in quantum affine gln} and Propoition~\ref{formula in Vz} that the matrix $S_\al=\sum_{1\leq i\leq n}\al_i\afE_{i,i+1}$ defines a semisimple representation of the cyclic quiver $\tri$.

\begin{Lem}\label{bar monomial} For $\al,\beta\in\afmbnn$, let $A=S_\al+\diag(\bt)\in\afThnr$. Then, in $\afSr$, $\ol{[A]}=[A]$ and $\ol{[\tA]}=[\tA]$. In particular,  we have
$\ol{S_\al(\bfl,r)}=S_\al(\bfl,r)$, $\ol{{}^t\!S_\al(\bfl,r)}={}^t\!S_\al(\bfl,r)$ for $\al\in\afmbnn$.
\end{Lem}
\begin{proof}
Let $\la=\ro(A)$ and $\mu=\co(A)$. Then, by
Lemma~\ref{longest element}, we have
$[A](C_{w_0,\mu}')=\ti T_{\fS_\la y_A\fS_\mu}$ and $[\tA](C_{w_{0,\la}}')=\ti T_{\fS_\mu y_{\tA}\fS_\la}$ (note that
$y_{\tA}=y_{A}^{-1}$).
By \cite[(2.0.2)]{DF13} (cf. the proof of \cite[Prop. 3.5]{DF13}), we have $y_A=\rho^{-\al_n}$ and $y_{\tA}=\rho^{\al_n}$.
It follows from \cite[(1.10)]{Curtis} that
$C'_{y_A^+}=\ti T_{\fS_\la y_A\fS_\mu}$ and $C'_{y_{\tA}^+}=\ti T_{\fS_\mu y_{\tA}\fS_\la}$. Thus,
\begin{equation*}
\begin{split}
\ol{[A]}(C_{w_0,\mu}')
&=\ol{[A](C_{w_0,\mu}')}=\ol{C'_{y_A^+}}=C'_{y_A^+}=[A](C_{w_0,\mu}')\\
\ol{[\tA]}(C_{w_0,\la}')&=\ol{[\tA](C_{w_0,\la}')}=\ol{C'_{y_{\tA}^+}}=C'_{y_{\tA}^+}
=[\tA](C_{w_{0,\la}}').
\end{split}
\end{equation*}
Consequently $\ol{[A]}=[A]$ and $\ol{[\tA]}=[\tA]$. The last assertion is clear.
\end{proof}

The stabilisation property developed at the beginning of last section gives the following stabilisation property.

\begin{Prop}\label{bar stab}
For $A\in\aftiThn$ there exist $C_1,\cdots,C_m\in\aftiThn$, elements $H_i(\up,\up')\in\sZ_1$ ($1\leq i\leq m$) and an integer $p_0\geq 0$ such that, in $ \afSpA$,
$$\ol{[{}_pA]}=\sum_{1\leq i\leq m}H_i(\up,\up^{-p})[{}_pC_i]\quad\text{ for all $p\geq p_0$}.$$
\end{Prop}
\begin{proof}
We prove the assertion by induction on $|\!|A|\!|$. If $|\!|A|\!|=0$ then $\ol{[{}_pA]}=[{}_pA]$ for all large enough $p$.
Assume now that $|\!|A|\!|\geq 1$ and the result is true for all $A'$ with $|\!|A'|\!|<|\!|A|\!|$.
By Lemma~\ref{stabilization property1} and Proposition~\ref{generalization of BLM 3.9}, there exist $A_i\in\aftiThn^{ss}$, $Z_j\in\aftiThn$ and $Q_j(\up,\up')\in\sZ_1$ ($1\leq i\leq N$, $1\leq j\leq m$) such that the following identity holds in
$\afSpA$
\begin{equation*}
[{}_pA]=[{}_{p}A_1]\cdots[{}_{p}A_N]-\sum_{1\leq j\leq m}Q_j(\up,\up^{-p})[{}_pZ_j]
\end{equation*}
for all large enough $p$, where $|\!|Z_i|\!|<|\!|A|\!|$ for $1\leq i\leq m$. It follows from Lemma~\ref{bar monomial} that
$$\ol{[{}_pA]}=[{}_{p}A_1]\cdots[{}_{p}A_N]-\sum_{1\leq j\leq m}\ol{Q_j(\up,\up^{-p})}\cdot\ol{[{}_pZ_j]}.$$
Now the assertion follows from the induction hypothesis.
\end{proof}

Recall the ring $\sZ_1$ defined in \eqref{sZ_1}. It admits a ring involution (i.e., a ring automorphism of order two) $\bar\,$ satisfying $\bar\up=\up^{-1}$ and $\bar{\up'}=\up^{\prime-1}$.  Extend the bar involution on $\sZ_1$ to
define a ring involution $\bar{}\,:\aftisK\ra\aftisK$
by setting
$\ol{A}=\sum_{1\leq i\leq m}H_i(\up,\up')C_i$ (notation of Proposition~\ref{bar stab}). This involution induces a ring involution \vspace{-1ex}
\begin{equation}\label{bar on afKn}
\bar\,:\afKn\ra\afKn\text{ which satisfies }\ol{\up^j\dob A\dcb}=\up^{-j}\sum_{1\leq i\leq m}H_i(\up,1)\dob C\dcb.
\end{equation}

The involution $\bar{}\,$ on $\afKn$ induces a $\mbq$-algebra involution $\bar\,:\afhbfKn\ra\afhbfKn$ such that $\ol{\sum_{A\in\aftiThn}\bt_A\dob A\dcb}=\sum_{A\in\aftiThn}\ol{\bt_A}\ol{\dob A\dcb}$.

\begin{Coro}\label{bar bimodule}
{\rm(1)} For $\al,\beta\in\afmbnn$, if $A=S_\al+\diag(\bt)\in\afThnr$, then $\ol{\dob A\dcb}=\dob A\dcb$ and $\ol{\dob \tA\dcb}=\dob\tA\dcb$. In particular,
for any $\al\in\afmbnn$, $\ol{S_\al\dop\bfl\dcp}=S_\al\dop\bfl\dcp$, $\ol{{}^t\!S_\al\dop\bfl\dcp}={}^t\!S_\al\dop\bfl\dcp$.

{\rm(2)}  There is a unique $\mbq$-algebra involution\footnote{This bar involution can also be induced from the bar involutions on $\afSr$ via $\afbfSn$ and $\afbfVn$. Thus, we may avoid using the stabilisation property.}
 $$\bar\ :\dbfHa\ra\dbfHa\text{
satisfying $\bar v=v^{-1}$, $\ol{\ti u_\la^\pm}=\ti u_\la^\pm$ and $\ol{K_i}=K_i^{-1}$ for  $\la\in\afmbnn$, $1\leq i\leq n$.}$$

{\rm(3)} The bar involution on $\afKn$ preserves the bimodule structure on $\afKn$.
\end{Coro}
\begin{proof} Clearly, by the definition of the bar involution on $\afKn$, (1) follows from Proposition~\ref{bar stab} and Lemma~\ref{bar monomial}. (2) follows from (1), Theorems~\ref{realization of dHa} and \ref{realization of ddHa}. Finally, (3) is clear as the bimodule structure on $\afKn$ is induced by the algebra structure of $\afhbfKn$ on which the bar involution is an ring automorphism.
\end{proof}

We first look at an algebraic construction of the canonical basis for affine quantum Schur algebras (see \cite{Lu99} for a geometric construction). We need the following interval finite condition.

\begin{Lem}\label{finite}
For $A\in\afThnpm$, the set $\{B\in\afThnpm\mid B\p A\}$  is finite. Hence, the intervals
$(-\infty, A']:=\{B\in\aftiThn\mid B\sqsubseteq A'\}$ for all  $A'\in\aftiThn$ are finite.
\end{Lem}
\begin{proof}
There exist $j_0\geq n$ such that $a_{s,j}=0$ for $1\leq s\leq n$ and $j\in\mbz$ with $|j|>j_0$. Let $\sX_A=\{B\in\afThnpm\mid b_{s,j}=0\text{ for } 1\leq s\leq n\text{ and }|j|>j_0,\,\sg(B)<|\!|A|\!|\}$. Then, $\sX_A$ is a finite set.
If $B\p A$, $1\leq i\leq n$ and $j_0<j$, then
\begin{equation*}
\begin{split}
b_{i,j}\leq\sg_{i,j}(B)\leq \sg_{i,j}(A)
&=\sum_{1\leq s\leq n\atop s<t,\,j\leq t}a_{s,t}|\{b\in\mbn\mid s-bn\leq i<j\leq t-bn\}|=0.
\end{split}
\end{equation*}
This implies that if $B\p A$, then $b_{i,j}=0$ for $1\leq i\leq n$ and $j>j_0$. Similarly, if $B\p A$, then $b_{i,j}=0$ for $1\leq i\leq n$ and $j<-j_0$. Furthermore, by \cite[3.7.6]{DDF}, we conclude that $\sg(B)\leq |\!|B|\!|<|\!|A|\!|$ for $B\in\afThnpm$ with $B\p A$. Consequently,
$\{B\in\afThnpm\mid B\p A\}\han\sX_A$, proving the first assertion. The last assertion is clear from \eqref{order sqsubset}.
\end{proof}

\begin{Prop}\label{canonical basis for affine q-Schur algebras}
$(1)$
There is a unique $\sZ$-basis $\{\th_{\Ar}\mid A\in\afThnr\}$ for $\afSr$ such that $\ol{\th_{\Ar}}=\th_{\Ar}$ and
\begin{equation}\label{th_Ar}
\th_{A,r}-[A]=\sum_{B\in\afThnr\atop B\sqsubset A}g_{B,A,r}[B]\in\sum_{B\in\afThnr\atop B\sqsubset A}\up^{-1}\mbz[\up^{-1}][B].
\end{equation}

$(2)$ For the canonical basis $\{A\}$, $A\in\afThnr$, of $\afSr$ defined in \cite[4.1(d)]{Lu99}, we have $\{A\}=\th_{A,r}$. In particular, $g_{B,A,r}$ can be described in terms of Kazhdan--Lusztig polynomials.
\end{Prop}
\begin{proof}
By Proposition~\ref{tri Hall}, for each $A\in\afThnr$, we may choose words $w_{A^+}\in\ti\Sg$ such
that  \eqref{eq tri Hall} hold. Let $w_{A^-}={}^t\!w_{{}^t\!\!A^-}$.
 By \eqref{eq tri Hall} and its opposite version for  $\ti u^-_{w_{A^-}}=\tau(\ti u_{{}^t\!(A^-)}^+)$ (see \eqref{tau}) and Proposition~\ref{tri-affine Schur algebras}, we have
 \begin{equation}\label{monomial tri}
\mpm^{(A)}:=\zr(\ti u^+_{w_{A^+}})[\diag(\bfsg(A))]
\zr(\ti u^-_{w_{A^-}})=[A]+\sum_{B\sqsubset A\atop B\in\afThnr}h_{A,B}[B]\, \quad(h_{A,B}\in\sZ).
\end{equation}
Now the interval finite condition in  Lemma~\ref{finite} implies that there exist $h_{A,B}'\in\sZ$ such that
$$[A]=\mpm^{(A)}+\sum_{B\in\afThnr\atop B\sqsubset A}h_{A,B}'\mpm^{(B)}.$$
Furthermore, by Lemma~\ref{bar monomial}, we have
$\ol{\mpm^{(A)}}=\mpm^{(A)}$ for $A\in\afThnr$.
Thus, \eqref{monomial tri} implies
$$\ol{[A]}=
\mpm^{(A)}+\sum_{B\in\afThnr\atop B\sqsubset A}\ol{h_{A,B}'}\mpm^{(B)}
=[A]+\sum_{B\in\afThnr\atop B\sqsubset A}k_{A,B}[B],$$
where $k_{A,B}\in\sZ$. Now (1) follows from a standard argument; see, e.g.,
\cite[7.10]{Lu90b}. Let $\leq$ be the partial order  on $\afThnr$ defined in \cite[4.1]{Lu99}. According to \cite[\S7]{Mcgerty}, if $A,B\in\afThnr$ and $B< A$ then $B\sqsubset A$.
Thus, by \cite[4.1(e)]{Lu99} and \cite[Remark 7.6]{VV99}, we conclude (2).
\end{proof}

We now construct the canonical basis for $\afKn$ as follows. See \cite{Fu2} for a construction in the non-affine case.

\begin{Thm}\label{canonical basis for afKn}
{\rm(1)} There exists a unique $\sZ$-basis $\{\th_A\mid A\in\aftiThn\}$ for $\afKn=\ddHa$ such that $\ol{\th_A}=\th_A$ and
$\th_{A}-\dob A\dcb\in\sum_{B\in\aftiThn,B\sqsubset A}v^{-1}\mbz[v^{-1}]\dob B\dcb$.

{\rm(2)} The algebra homomorphism $\dzr:\afKn\to\afSr$ given in \eqref{dzr([A])} preserves the bar involution and the canonical bases:
\vspace{-2ex}
$${\rm(a)}\; \dzr(\bar u)=\ol{\dzr(u)}\text{ for all } u\in\afKn;\quad\quad{\rm(b)}\;
\dzr(\th_A)=\begin{cases}\th_{\Ar}, &\text{ if } A\in\afThnr;\\
0, &\text{ otherwise.}
\end{cases}$$

{\rm(3)} There is an anti-automorphism $\dot\tau$ on $\afKn$ such that $\dot\tau(\dob A\dcb)=\dob\tA\dcb$ and $\dot\tau(\th_A)=\th_{\tA}$.
\end{Thm}
\begin{proof} Consider the monomial basis $\{\mpM^{(A)}\mid A\in \aftiThn\}$ given in Corollary~\ref{ddHa monomial basis}. Then Lemma~\ref{bar monomial} implies
$\ol{\mpM^{(A)}} =\mpM^{(A)}$
 and \eqref{tri afKn} together with the interval finite property Lemma~\ref{finite} implies
$\dob A\dcb=\mpM^{(A)}+h$,
where $h$ is a $\sZ$-linear combination of $\mpM^{(C)}$ with  $C\in\aftiThn$ and $C\sqsubset A$. Thus,  we conclude that $\ol{\dob A\dcb}-\dob A\dcb\in\sum_{C\in\aftiThn\atop
C\sqsubset A}\sZ\dob C\dcb.$ Hence, like the proof of Proposition~\ref{canonical basis for affine q-Schur algebras}, a standard argument proves (1).

According to \eqref{dzr([A])} and Lemma~\ref{bar monomial} we see that
$\dzr(\ol{\mpM^{(A)}})=\ol{\dzr({\mpM^{(A)}})}$ for $A\in\aftiThn$.
Furthermore, by Corollary~\ref{ddHa monomial basis}, the set $\{\mpM^{(A)}\mid A\in\aftiThn\}$ forms a $\sZ$-basis for $\afKn$. Thus, $\dzr(\bar u)=\ol{\dzr(u)}$ for $u\in\afKn$. The second assertion in (2) follows from the argument for the uniqueness of canonical basis.

By \cite[1.11]{Lu99}, the $\sZ$-linear map
$\tau_r:\afbfSr\ra\afbfSr$, $[A]\mapsto[\tA]$
is an algebra anti-automorphism, where $\tA$ is the transpose of $A$. By Proposition~\ref{stabilization property}, the maps $\tau_r$ induce an algebra anti-automorphism
$\dot\tau:\afKn\ra\afKn$ such that $\dot\tau(\dob A\dcb)=\dob\tA\dcb$ for $A\in\aftiThn$. Finally, applying $\dot\tau$ to $\th_{A}-\dob A\dcb$ yields $\dot\tau(\th_A)=\th_{\tA}$ by the uniqueness of canonical bases.
\end{proof}

\begin{Rem}
The basis constructed in Theorem~\ref{canonical basis for afKn}(1) is the canonical basis for the integral modified quantum affine $\mathfrak{gl}_n$. Theorem~\ref{canonical basis for afKn}(2b) shows that this basis is the lifting of the canonical bases for affine quantum Schur algebras. A similar basis with a similar property for the modified quantum affine $\mathfrak{sl}_n$ was conjectured by Lusztig in \cite[9.3]{Lu99}. This conjecture (rather its slight modified version) was proved by Vasserot and Schiffmann in \cite{SV}.
 Thus, Theorems \ref{unmodified}, \ref{realization of ddHa} and \ref{canonical basis for afKn} can be regarded as
 of a generalisation of the conjecture of Lusztig to the quantum loop algebra $\bfU(\afgl)$.
We will address an extension of our approach to the extended quantum affine $\mathfrak{sl}_n$ case in the last section.
\end{Rem}


We end this section with a comparison of this canonical basis and the canonical basis for the Ringel--Hall algebra of a cyclic quiver.
According to \cite[Prop 7.5]{VV99} (see also \cite{LTV}), there is a unique $\sZ$-basis $\{\th_A^+\mid A\in\afThnp\}$ for the Ringel--Hall algebra $\Ha=\dHap$ such that $\ol{\th_A^+}=\th_A^+$ and
\begin{equation}\label{aaa}
\th_A^+-\ti u_A^+\in\sum_{B\p A,\,B\in\afThnp\atop
\bfd(B)=\bfd(A)}\up^{-1}\mbz[\up^{-1}]\ti u_B^+.
\end{equation}
\begin{Prop}
Assume $A\in\afThnp$ and $\la\in\afmbzn$. Then we have $\th_{A}^+\dob \diag(\la)\dcb=\th_{A+\diag(\la-co(A))}$. In particular, we have
$\th_A^+=\sum_{\mu\in\afmbzn}\th_{A+\diag(\mu)}$.
\end{Prop}
\begin{proof} By \eqref{bimodule} and \eqref{aaa},
$$\th_A^+\dob\diag(\la)\dcb-\dob A+\diag(\la-co(A))\dcb\in\sum_{B\in\afThnp, B\prec A\atop
\bfd(B)=\bfd(A)}\up^{-1}\mbz[\up^{-1}]\dob B+\diag(\la-co(B))\dcb.$$
It is direct to check that, for $\bfd(B)=\bfd(A)$ and $B\in\afThnp$, $\ro(B)-\co(B)=\ro(A)-\co(A)$. Hence,
$$\th_A^+\dob\diag(\la)\dcb-\dob A+\diag(\la-co(A))\dcb\in\sum_{C\in\aftiThn \atop
C\sqsubset A+\diag(\la-\co(A))}\up^{-1}\mbz[\up^{-1}]\dob C\dcb.$$
Also, by Corollary~\ref{bar bimodule}(3), $\ol{\th_A^+\dob\diag(\la)\dcb}=\ol{\th_A^+}\,\ol{\dob\diag(\la)\dcb}=\th_A^+\dob\diag(\la)\dcb$.
Hence, the first assertion follows from the uniqueness of the canonical basis. Now, the identity element $1=\sum_{\la\in\afmbzn}\dob \diag(\la)\dcb$ gives the last assertion.
\end{proof}

\section{Application to a conjecture of Lusztig}

Let $\bfU_\vtg(n)$ be the extended affine $\mathfrak{sl}_n$ as defined in Theorem~\ref{presentation-dbfHa}(2) and let $\dbfpU=\oplus_{\la,\mu\in\afmbzn}\bfU_\vtg(n)/{}_\la{'\!}I_\mu$, where
${}_\la {'\!I}_\mu:=\sum_{\bfj\in\afmbzn}(\Kbfj-
 \up^{\la\cdot\bfj})\bfpU+\sum_{\bfj\in\afmbzn}\bfpU(\Kbfj
 -\up^{\mu\cdot\bfj})$. Since ${}_\la {'\!I}_\mu={}_\la I_\mu\cap\bfpU$ (see Theorem~\ref{presentation-dbfHa}(2)), it follows that $\dbfpU\cong\oplus_{\la,\mu\in\afmbzn}{}_\la\bfpU_\mu$, where $_\la\bfpU_\mu=\pi_{\la,\mu}(\bfpU)$. Thus, we will regard $\dbfpU$ as this subalgebra of $\ddbfHa=\afbfKn$.
We now look at an application to the conjecture given in \cite[9.3]{Lu99} which is proved in \cite{SV}.

Let  $\dUn$ be the $\sZ$-subalgebra of $\ddbfHa$ generated by
$$\ti u_{m\afbse_i}^+[\diag(\la)]=E_i^{(m)}\dob\diag(\la)\dcb, \quad \ti u_{m\afbse_i}^-[\diag(\la)]=F_i^{(m)}\dob\diag(\la)\dcb$$
for all $1\leq i\leq n$, $m\in\mbn$ and $\la\in\afmbzn$. Then $\dUn$ is a  subalgebra of $\ddHa=\afKn$.

Call a matrix $A=(a_{i,j})\in\aftiThn$ to be {\it aperiodic} if for every integer $l\neq0$ there exists $1\leq i\leq n$ such that $a_{i,i+l}=0$. Let
$\aftiThnap$ be the set of all aperiodic matrices in $\aftiThn$.

Recall the monomial basis for $\ddHa$ given in Corollary~\ref{ddHa monomial basis}.

\begin{Lem}\label{dot mono basis} The set $\{\mpM^{(A)}\mid A\in \aftiThnap\}$ forms a $\sZ$-basis for $\dUn$.
\end{Lem}
\begin{proof} By {\cite[Th.~7.5(1)]{DDX}}, the elements $\ti u^+_{(w_{A})}$, $A\in\afThnp\cap\aftiThnap$, form a basis for the $+$-part $U_\vtg^+(n)_\sZ$ generated by all $E_i^{(m)}$. Hence,  the set $\{\mpM^{(A)}\mid A\in \aftiThnap\}$ spans $\dUn$. By \eqref{tri afKn}, the set is linearly independent.
\end{proof}
For each $A\in\aftiThnap$,  use the coefficients $h_{A,B}$ given in \eqref{tri afKn} and the order $\sqsubseteq$ given in \eqref{order sqsubset} to
define (cf. \cite[Def.~7.2]{DDX}) recursively the elements $\mpE_A\in\dUn$ by
\begin{equation}\label{PA}
\mpE_A=\begin{cases}\mpM^{(A)},\quad&\text{ if $A$ is minimal relative to }\sqsubseteq;\\
\mpM^{(A)}-\sum_{B\sqsubset A\atop B\in\aftiThnap}h_{A,B}\mpE_B. &\text{ otherwise.}
\end{cases}
\end{equation}

\begin{Lem}\label{mpPA-[A]} {\rm(1)} The set $\{\mpE_A\mid A\in \aftiThnap\}$ forms a $\sZ$-basis for $\dUn$.

{\rm(2)} For $A\in\aftiThnap$ we have
$\mpE_A-[A]\in\sum_{B\in\aftiThn\backslash\aftiThnap\atop B\sqsubset A}\sZ[B].$

\end{Lem}
\begin{proof} Statement (1) follows from Lemma \ref{dot mono basis} and the definition $\mpE_A$ \eqref{PA}.
We prove (2) by induction on $|\!|A|\!|$. The assertion is clear for by $|\!|A|\!|=0$. Assume now $|\!|A|\!|\geq 1$
By \eqref{tri afKn} and \eqref{PA}, we have
$$\mpE_A-\dob A\dcb+\sum_{B\in\aftiThnap\atop B\sqsubset A}h_{A,B}(\mpE_B-\dob B\dcb)=\sum_{B\in\aftiThn\backslash\aftiThnap\atop B\sqsubset A}h_{A,B}\dob B\dcb.$$
Now the assertion follows from induction since $B\sqsubset A$ implies $\ddet{B}<\ddet{A}$.
\end{proof}

Note that the restriction of the bar involution \eqref{bar on afKn} gives a bar involution on $\dUn$.

\begin{Prop}\label{thA'}
There exists a unique $\sZ$-basis $\{\th_{A}' \mid A\in\aftiThnap \}$ for $\dUn$ such that $\ol{\th_A'}=\th_A'$ and
$$\th_A'-\mpE_A\in\sum_{B\in\aftiThnap,\,B\sqsubset A}\up^{-1}\mbz[\up^{-1}]\mpE_B.$$
\end{Prop}
\begin{proof} Since, by \eqref{PA},
$$\mpE_A=\mpM^{(A)}+\text{a $\sZ$-linear combination of $\mpM^{(C)}$ with $C\in\aftiThnap$ and $C\sqsubset A$},$$ it follows that $\ol{\mpE_A}-\mpE_A\in\displaystyle\sum_{C\in\aftiThnap\atop C\sqsubset A}\sZ\mpE_C.$
Now the assertion follows from a standard argument.
\end{proof}

\begin{Rem} Motivated by \cite[Th.~8.2]{Lu99}, it would be natural to conjecture that $\theta_A\in\dUn$ for all $A\in\aftiThnap$. Equivalently, $\theta_A'=\theta_A$ if $A\in\aftiThnap$ (cf. \cite[Th.~8.5]{DDX}). In the rest of the section, we show some strong evidence for the truth of this conjecture.
\end{Rem}
Let $\sL_r=\sum_{A\in\afThnr}\mbz[\up^{-1}][A]\in\afSr$ and let $\msP$ be the $\sZ$-submodule of $\ddHa$ spanned by the {\it periodic} elements $\dob B\dcb$ with $B\in\aftiThn\backslash\aftiThnap$.
Recall the algebra homomorphisms $\zeta$ in Theorem~\ref{realization} and $\dzr$ in \eqref{dzr([A])} and note that $\dzr(\msP)\cap\bfU_\vtg(n,r)=0$, where $\zeta_r(\bfU_\vtg(n))=\bfU_\vtg(n,r)$.

Let $\afThnrap=\aftiThnap\cap\afThnr$.

\begin{Lem}\label{dzr(mpPA-[A])}
 Assume $A\in\aftiThnap$.
\begin{itemize}
\item[(1)] If  $A\not\in\afThnr$ then we have
$\dzr(\mpE_A)=0$.

\item[(2)] If $A\in\afThnr$ then we have $\dzr(\mpE_A)-[A]\in\up^{-1}\sL_r$.
\end{itemize}
\end{Lem}
\begin{proof}
If  $A\not\in\afThnr$, Lemma \ref{mpPA-[A]}(2) implies
$\dzr(\mpE_A)=\dzr(\mpE_A)-\dzr([A])\in\dzr(\msP)\cap\bfU_\vtg(n,r)=0$, proving (1).

Now we assume $A\in\afThnr$.
If $|\!|A|\!|=0$ then $\mpE_A=\dob A\dcb$ and $\dzr(\mpE_A)-[A]=0$. Now we assume $|\!|A|\!|>0$. We write $\th_{A,r}$ as in \eqref{th_Ar}.
 By Lemma~\ref{mpPA-[A]} and \cite[8.2]{Lu99}, we see that
\begin{equation*}
\begin{split}
\th_{A,r}-\bigg(\dzr(\mpE_A)+\sum_{B\in\afThnrap\atop B\sqsubset A}g_{B,A,r}\dzr(\mpE_B)\bigg)&=([A]-\dzr(\mpE_A))+\sum_{B\in\afThnrap\atop B\sqsubset A}g_{B,A,r}([B]-\dzr(\mpE_B))\\
&\quad+\sum_{B\in\afThnr\backslash\afThnrap\atop B\sqsubset A}g_{B,A,r}[B],
\end{split}
\end{equation*}
which belongs to $\dzr(\msP)\cap\bfU_\vtg(n,r)=0.$ Thus, by the induction hypothesis,
$$\dzr(\mpE_A)-[A]=\sum_{B\in\afThnrap\atop B\sqsubset A}g_{B,A,r}([B]-\dzr(\mpE_B))+\sum_{B\in\afThnr\backslash\afThnrap\atop B\sqsubset A}g_{B,A,r}[B]\in\up^{-1}\sL_r$$
as required.
\end{proof}

We now show that the basis $\theta_A'$ satisfies a property similar to Theorem \ref{canonical basis for afKn}(2b).

\begin{Thm}\label{Lconjecture}
Let $A\in\aftiThnap$. Then we have
$$\dzr(\th_A')=
\begin{cases}\th_{A,r}&\text{if $A\in\afThnr;$}\\
0&\text{if $A\not\in\afThnr$.}
\end{cases}$$
Hence, we have $\dzr(\th_A')=\dzr(\th_A)$ for  $A\in\aftiThnap$.
\end{Thm}
\begin{proof}
If $A\not\in\afThnr$ then, by Proposition~\ref{thA'} and Lemma~\ref{dzr(mpPA-[A])}, we see that
$$\dzr(\th_A')=\dzr(\th_A'-\mpE_A)\in\sum_{B\in\afThnrap\atop B\sqsubset A}\up^{-1}\mbz[\up^{-1}]\dzr(\mpE_B)\han\up^{-1}\sL_r.$$
If $A\in\afThnr$ then, by loc. cit., we have
$$\dzr(\th_A')\in\dzr(\mpE_A)+\sum_{B\in\afThnrap\atop B\sqsubset A}\up^{-1}\mbz[\up^{-1}]\dzr(\mpE_B)\han [A]+\up^{-1}\sL_r.$$
Furthermore, we have $\ol{\dzr(\th_A')}=\dzr(\th_A')$ for all $A\in\aftiThnap$. The assertion follows the uniqueness of the canonical basis.
\end{proof}

Theorem \ref{Lconjecture} gives an algebraic construction of the conjecture of Lusztig stated at the end of \cite[\S9.3]{Lu99}\footnote{This conjecture was made for quantum affine $\mathfrak{sl}_n$ with associated modified quantum group idempotented on $\mbz^{n-1}$.} for the modified extended quantum affine $\mathfrak{sl}_n$, $\dUn$, idempotented on $\mbz^{n}$; see \cite{SV} for a proof for the (polynomial weighted) modified quantum affine $\mathfrak{sl}_n$ which is idempotented on $\mbn^{n}$ (compare the construction in \cite[\S7]{Mcgerty} for the modified quantum affine $\mathfrak{sl}_n$ idempotented on $\mbz^{n-1}$). Note that, by the presentation for $\dUn$ given in \cite[31.1.3]{Lubk}, this modified algebra of Schiffmann--Vasserot is a homomorphic image of $\dUn$.

\end{document}